\documentclass{scrartcl}
\usepackage[colorlinks=true,citecolor=blue]{hyperref}
\usepackage{mathptmx, amsmath, amssymb, amsfonts, amsthm, mathptmx, enumerate, color}
\usepackage{xcolor}
\usepackage{pgf}
\usepackage[left=2.5cm,right=2.5cm,top=1.5cm,bottom=1.5cm,includeheadfoot]{geometry}
\usepackage[singlespacing]{setspace}

\newtheorem{theorem}{Theorem}[section]
\newtheorem{corollary}{Corollary}[section]
\newtheorem{lemma}{Lemma}[section]
\newtheorem{proposition}{Proposition}[section]
\theoremstyle{definition}
\newtheorem{definition}{Definition}[section]
\newtheorem{example}{Example}[section]
\newtheorem{remark}{Remark}[section]

\numberwithin{equation}{section}

\let\origitemize\itemize
\def\itemize{\origitemize\itemsep0pt}
\numberwithin{equation}{section}

\makeatletter
\DeclareOldFontCommand{\rm}{\normalfont\rmfamily}{\mathrm}
\DeclareOldFontCommand{\sf}{\normalfont\sffamily}{\mathsf}
\DeclareOldFontCommand{\tt}{\normalfont\ttfamily}{\mathtt}
\DeclareOldFontCommand{\bf}{\normalfont\bfseries}{\mathbf}
\DeclareOldFontCommand{\it}{\normalfont\itshape}{\mathit}
\DeclareOldFontCommand{\sl}{\normalfont\slshape}{\@nomath\sl}
\DeclareOldFontCommand{\sc}{\normalfont\scshape}{\@nomath\sc}
\makeatother
\addtokomafont{disposition}{\rmfamily}
\begin{document}

\title{
	Nonlinear Strict Cone Separation Theorems in Real Normed Spaces
}

\author{
Christian G\"unther\thanks{Leibniz Universität Hannover, Institut f\"ur Angewandte Mathematik, Welfengarten 1, 30167 
		Hannover, Germany
		(\href{mailto:c.guenther@ifam.uni-hannover.de}{c.guenther@ifam.uni-hannover.de}).}
	\and Bahareh Khazayel\thanks{Martin Luther University Halle-Wittenberg, Faculty
		of Natural Sciences II, Institute of Mathematics, 06099 Halle (Saale), Germany
		(\href{mailto:Bahareh.Khazayel@mathematik.uni-halle.de}{Bahareh.Khazayel@mathematik.uni-halle.de}).}
	\and Christiane Tammer\thanks{Martin Luther University Halle-Wittenberg, Faculty
		of Natural Sciences II, Institute of Mathematics, 06099 Halle (Saale), Germany
		(\href{mailto:Christiane.Tammer@mathematik.uni-halle.de}{Christiane.Tammer@mathematik.uni-halle.de}).}}

\maketitle

\begin{center}
{\large This article is dedicated to Professor Alfred G\"opfert.}\\[0.5cm]
	\textbf{Abstract}
\end{center}
\begin{abstract}
	In this paper, we derive some new results for the separation of two not necessarily convex cones by  a (convex) cone / conical surface in real (reflexive) normed spaces. In essence, we follow the nonlinear and nonsymmetric separation approach developed by 
    Kasimbeyli (2010, SIAM J. Optim. 20), which is based on augmented dual cones and Bishop-Phelps type (normlinear) separating functions.
    Compared to Kasimbeyli's separation theorem, we formulate our theorems for the separation of two cones under weaker conditions (concerning convexity and closedness requirements) with respect to the involved cones. By a new characterization of the algebraic interior of augmented dual cones in real normed spaces, we are able to establish relationships between our cone separation results and the results derived by Kasimbeyli (2010, SIAM J. Optim. 20) and by Garc\'{i}a-Casta\~{n}o, Melguizo-Padial and Parzanese (2023, Math. Meth. Oper. Res. 97). 
\end{abstract}


\begin{flushleft}
\end{flushleft}

\begin{flushleft}
	\textbf{Mathematics subject classifications (MSC 2010):} 46A22; 49J27; 65K10; 90C48 
\end{flushleft}

\section{Introduction}
Classical separation arguments from Convex Analysis are main tools in optimization theory, especially for deriving duality assertions and optimality conditions (see G\"opfert \cite{Goe1973}, G\"opfert {\rm et al.} \cite{GoeRiaTamZal2023}, Jahn \cite{Jahn2011}, Rockafellar \cite{Roc1970}, Z{\u{a}}linescu \cite{Zal2002}  and references therein).
Cone separation theorems, i.e., theorems related to the separation of two cones by a hyperplane or a certain conical surface, are studied by several authors in the literature, among others by Ad\'{a}n and Novo \cite[Th. 2.1]{AdanNovo2004}, G\"opfert {\rm et al.} \cite[Sec. 2.4.3]{GoeRiaTamZal2023}, G\"unther, Khazayel and Tammer \cite{GuenKhaTam22}, Henig \cite{Henig}, Jahn \cite[Th. 3.22]{Jahn2011}, \cite[Sec. 3.7]{Jahn2023}, Kasimbeyli \cite[Th. 4.3]{Kasimbeyli2010}, Khazayel {\rm et al.} \cite[Sec. 2.3]{Khazayel2021a}, Nieuwenhuis \cite{Nieuwenhuis}, and Novo and Z\u{a}linescu \cite[Cor. 2.3]{NovoZali2021}.
It is well known that such cone separation theorems are useful in certain fields of optimization (for instance, for deriving scalarization results for nonconvex vector optimization problems; see, e.g., Bo{\c{t}}, Grad and Wanka \cite{BotGradWanka}, Garc\'{i}a-Casta\~{n}o, Melguizo-Padial and Parzanese \cite{CastanoEtAl2023}, Gerth and Weidner \cite{GerthWeidner}, Eichfelder and Kasimbeyli \cite{EichfelderKasimbeyli2014}, G\"opfert {\rm et al.} \cite{GoeRiaTamZal2023}, Grad \cite{Grad2015}, Kasimbeyli \cite{Kasimbeyli2010}, \cite{Kasimbeyli2013},  Kasimbeyli {\rm et al.} \cite{Kasimbeyli2019}, and  Tammer and Weidner \cite{TammerWeidner2020}). 

In the literature, there are different concepts for the nonlinear separation of two cones by a cone/conical surface (see, e.g., G\"unther, Khazayel and Tammer \cite{GuenKhaTam22}, Henig \cite{Henig}, Kasimbeyli \cite{Kasimbeyli2010}, Nehse \cite{Nehse1981}, Nieuwenhuis \cite{Nieuwenhuis}). In this paper, we adapt two separation concepts recently used by G\"unther, Khazayel and Tammer \cite{GuenKhaTam22} for separating two (not necessarily convex) cones by a cone/conical surface. Consider two nontrivial cones $\Omega^1, \Omega^2 \subseteq E$ and a (closed, convex) cone $C \subseteq E$ with nonempty interior (in what follows ${\rm int}\, C$, ${\rm bd}\, C$,  ${\rm cl}\, C$ stand for the interior, boundary, closure of $C$) in a real normed space $(E, ||\cdot||)$. Then, the cones $\Omega^1$ and $\Omega^2$ are said to be \\
\begin{itemize}
    \item[$\bullet$] { (weakly) separated by (the boundary of) the cone $C$} if 
    $$
     \Omega^1 \subseteq {\rm cl}\, C \quad \mbox{and} \quad \Omega^2 \subseteq E \setminus {\rm int}\, C.   
    $$
    \item[$\bullet$] { strictly separated by (the boundary of) the cone $C$} if 
    $$
    \Omega^1 \setminus \{0\} \subseteq {\rm int}\, C \quad \mbox{and} \quad \Omega^2 \setminus \{0\} \subseteq E \setminus {\rm cl}\, C.
    $$
\end{itemize}

It is easy to check that strict separation also ensures weak separation. Assume that $C$ is a cone in $E$ with nonempty interior. Then, $\Omega^1$ and $\Omega^2$ are (weakly, strictly) separated by the cone $C$ if and only if $\Omega^2$ and $\Omega^1$ are (weakly, strictly) separated by the cone $E \setminus ({\rm int}\, C)$. By imposing convexity of $C$ in the separation concept, one is loosing this symmetry property of the  approach, since $E \setminus ({\rm int}\, C)$ is not necessarily a convex cone. 

\begin{example}
    Consider $E = \mathbb{R}^2$, $\Omega^1 = \mathbb{R}^2_+$ and $\Omega^2 = E \setminus ({\rm int}\,\mathbb{R}^2_+)$. Then, $\Omega^1$ and $\Omega^2$ are weakly separated by (the boundary of the) cone $C = \Omega^1$ (and there is actually no other separating cone). However, there does not exist a convex cone such that $\Omega^2$ and $\Omega^1$ are weakly separated (since the smallest convex set containing $\Omega^2$ is the whole space $\mathbb{R}^2$).  Similarly, one can see that the symmetry property of the strict separation approach fails if one assumes that $C$ is convex in the separation concept.  
\end{example}

Notice that cones from the set
$$
\{C \subseteq E \mid C \text{ is a convex cone with } \Omega^1 \setminus \{0\} \subseteq {\rm int}\, C\}
$$
are known in the vector optimization / order theory literature as dilating cones for $\Omega^1$  (Henig \cite{Henig2} used such dilating cones to define a proper efficiency solution concept in vector optimization; see also Grad \cite[Ch. 3]{Grad2015}, G\"unther, Khazayel and Tammer \cite{Khazayel2021b}, and Khan, Tammer and Z{\u{a}}linescu \cite[Sec. 2.4]{Khanetal2015}).  In our nonsymmetric strict separation approach, the separating cone $C$ is always a dilating cone for the first given cone $\Omega^1$.  Knowing for which cone a dilation cone is obtained is useful from the point of view of applications in vector optimization (for instance in proper efficiency solution concepts). 

\begin{remark}
Furthermore, the following fact are easy to observe:
\begin{itemize}
\item[$\bullet$] If $C$ is a convex cone, then 
$\Omega^1$ and $\Omega^2$ are (weakly, strictly) separated by  $C$ if and only if ${\rm conv}(\Omega^1)$ and $\Omega^2$ are (weakly, strictly) separated by  $C$.
\item[$\bullet$] If $C$ is a closed, convex cone, then 
$\Omega^1$ and $\Omega^2$ are weakly separated by $C$ if and only if $\widehat \Omega^1  := {\rm cl}({\rm conv}\,\Omega^1) $ and $\Omega^2$ are weakly separated by $C$.
\item[$\bullet$] Strict separation of $\Omega^1$ and $\Omega^2$ by $C$ ensures $\Omega^1 \subseteq {\rm cl}\, C$ and $\Omega^2 \setminus \{0\} \subseteq E \setminus {\rm cl}\, C$, hence $\Omega^1 \cap \Omega^2 = \{0\}$ is a necessary  condition for strict separation.
\item[$\bullet$] 
If $C$ is closed and convex, then $\widehat \Omega^1\subseteq {\rm cl}\, C$, and so $\widehat \Omega^1 \cap \Omega^2 = \{0\}$ is necessary for strict separation. 
\item[$\bullet$] 
Assume that ${\rm cl}\, C$ is pointed. Then, the pointedness of $\Omega^1$ is necessary for weak separation of $\Omega^1$ and $\Omega^2$ (since $\Omega^1 \subseteq {\rm cl}\, C$). However, $\Omega^2$ can be a not pointed cone, and so the two cones to be separated do not play symmetric roles in our approach. 
\item[$\bullet$] If $C$ is pointed, closed and convex, then the pointedness of $\widehat \Omega^1$ is necessary for (weak, strict) separation of $\Omega^1$ and $\Omega^2$.

\end{itemize}
\end{remark}
 
In our upcoming cone separation theorems, we like to consider a (pointed, closed and convex) cone $C$ that can be represented by the sublevel set (w.r.t. the level 0) of a (lower semicontinuous, subadditive, convex, positively homogeneous) function $\varphi: E \to \mathbb{R}$. More precisely, $\varphi$ should satisfy the cone representation properties (see Jahn \cite{Jahn2023})
$$
{\rm cl}\, C = \{x \in E \mid \varphi(x) \leq 0 \} \quad \mbox{and} \quad {\rm int}\, C = \{x \in E \mid \varphi(x) < 0 \}.
$$
In the work by Kasimbeyli \cite{Kasimbeyli2010} (compare also \cite{AcaKas21}, \cite{BagAdi13}, \cite{CastanoEtAl2023}, \cite{GasOzt06}, \cite{GuenKhaTam22}, \cite{OztGur15}), the separating function $\varphi$ is given by a Bishop-Phelps type (normlinear) function $\varphi_{x^*, \alpha}: E \to \mathbb{R}$ (where some linear continuous function $x^*$ in the topological dual space $E^*$ of $E$, the norm $\mid \mid \cdot \mid \mid $ on $E$, and some $\alpha \geq 0$ are involved), which is defined by
\begin{equation}\label{f12}
\varphi_{x^*, \alpha}(x)  := x^*(x) + \alpha ||x|| \quad \mbox{for all } x \in E.
\end{equation}
The type of function $\varphi_{x^*, \alpha}$ in \eqref{f12} is associated to so-called Bishop-Phelps cones and dates back to the work by Bishop and Phelps \cite{BP1962}.  Recall, for any $y^* \in E^*$, the well known formulation of a Bishop-Phelps cone \cite{BP1962} (see also Ha and Jahn \cite[Rem. 2.1]{HaJahn2021}) is given by
$$
C(y^*) := \{ x \in E \mid y^*(x) \geq ||x||\}.
$$
In particular, if $\alpha > 0$, then 
$$C(-\alpha^{-1} x^*) = \{ x \in E \mid \varphi_{x^*, \alpha}(x) \leq 0\}.$$
It is known that Bishop-Phelps cones have a lot of useful properties and there are interesting applications in variational analysis and nonconvex optimization (see Phelps \cite{Phe93}) as well as  in vector optimization (see, e.g., Eichfelder \cite{Eichfelder2014}, Eichfelder and Kasimbeyli \cite{EichfelderKasimbeyli2014}, Ha \cite{Ha2022}, Ha and Jahn \cite{HaJahn2017}, \cite{HaJahn2021}, Jahn \cite{Jahn2009}, \cite[p. 159-160]{Jahn2011},  Kasimbeyli \cite{Kasimbeyli2010}, Kasimbeyli and Kasimbeyli \cite{KasKas17}). 
It is known that $C(y^*)$ is a closed, pointed, convex cone. If $||y^*||_* > 1$ (where $||\cdot||_*$ denotes the dual norm of $||\cdot||$), then $C(y^*)$ is nontrivial and 
$$C_>(y^*) := \{ x \in E \mid y^*(x) > ||x||\} = {\rm int}\, C(y^*) \neq \emptyset.$$
Notice, $\varphi_{x^*, \alpha}$ in \eqref{f12} is also known as a ``normlinear function'', since it is a combination of a norm function with a linear (continuous) function (see the recent paper by Zaffaroni \cite{Zaffaroni2022} with $x^* \in E^*$ and $\alpha \in \mathbb{R}$ in \eqref{f12}). 

\bigskip

The paper is structured as follows. In Section \ref{sec:not_and_pre}, we present some preliminaries (related  to cones, bases, dual cones, classical linear separation theorems) in real normed spaces. Theorem \ref{th:properties_top_dual_cones} provides useful characterizations of the (algebraic) interior of the dual cone of a nontrivial cone. 

Section \ref{sec:augmented_dual_cones} is devoted to the study of augmented dual cones and their properties. Theorem \ref{th:0notinclSK} discusses the nonemptyness of certain subsets of the augmented dual cone, while Theorem \ref{th:solidness_augmented_dual_cones} gives new characterizations of the algebraic interior of the augmented dual cone of a nontrivial cone in real (reflexive) normed spaces.

In Section \ref{sec:nonSepAppAndBPCones}, we present our nonlinear separation approach based on Bishop-Phelps cones. The main role plays the nonlinear separating function given in \eqref{f12}.

Our main nonlinear strict cone separation theorems (namely Theorems \ref{th:sep_cones_Banach_1} and \ref{th:sep_cones_Banach_3}) are given in Section \ref{sec:cone_separation_normed}. As a direct consequence of our results, we also get a separation theorem (see Proposition \ref{cor:sep_main_top_kasimbeyli}) which is similar to the one derived by Kasimbeyli \cite{Kasimbeyli2010}. With the aid of the characterizations derived in Theorem \ref{th:solidness_augmented_dual_cones}, we are able to establish relationships between our cone separation results and the results derived by Kasimbeyli \cite{Kasimbeyli2010} and by Garc\'{i}a-Casta\~{n}o, Melguizo-Padial and Parzanese \cite{CastanoEtAl2023}.

Section \ref{sec:conclusion} provides some concluding remarks and an outlook for further research.

\section{Notations and Preliminaries} \label{sec:not_and_pre}

\subsection{Fundamentals of Real Normed Spaces}	
We consider a real normed space $(E, ||\cdot||)$ throughout this article. As usual,
$$
E^* = (E, ||\cdot||)^* = \{x^*: E \to \mathbb{R} \mid x^* \mbox{ is linear and continuous}\}
$$
denotes the topological dual space of $E$. 
It is known that $E^*$ is a real Banach space endowed with the dual norm $||\cdot||_*: E^* \to \mathbb{R}$ of the norm $||\cdot||$, which is defined by 
$$||x^*||_* := \sup_{||x|| \leq 1} |x^*(x)| \quad \text{for all } x^* \in E^*.$$
For the underlying strong topology $\tau$ on $E$ we will use the topology that is induced by the norm $||\cdot||$ (known as norm topology).
The so-called weak topology on $E$, that we denote by $\tau_w := \sigma(E, E^*)$, is the coarsest topology on $E$ making all the seminorms $\{\psi_{x^*} \mid x^* \in E^*\}$ continuous, where $\psi_{x^*}: E \to \mathbb{R}$ is defined by 
$
\psi_{x^*}(x) := |x^*(x)|$ for all $x \in E$. Notice that we have $(E, ||\cdot||)^* = (E, \tau_w)^*$.
Beside the norm topology in $E^*$ one can also consider the so-called weak$^*$ topology on $E^*$ that we denote by $\tau_{w^*} := \sigma(E^*, E)$ (see, e.g., Jahn \cite[Def. 1.38 (c)]{Jahn2011}).

Consistently, the set of nonnegative real numbers is denoted by $\mathbb{R}_+$, while $\mathbb{P} := \mathbb{R}_{++}$ denotes the set of positive real numbers. 
Given a set $\Omega \subseteq E$, we denote by ${\rm cl}\,\Omega$, ${\rm bd}\,\Omega$, and ${\rm int}\,\Omega$ the closure, the boundary, and the interior of $\Omega$ w.r.t. the norm topology. 
As usual, the algebraic interior (the core) of $\Omega$ is defined by
$$
{\rm cor}\, \Omega := \{x \in \Omega \mid \forall\, v \in E\; \exists\, \varepsilon > 0: \; x + [0, \varepsilon] \cdot v \subseteq \Omega\}.
$$
It is known that
\begin{align*}
{\rm int}\, \Omega \subseteq {\rm cor}\, \Omega \subseteq \Omega  \subseteq  {\rm cl}\, \Omega = ({\rm int}\, \Omega) \cup {\rm bd}\, \Omega.
\end{align*}

We bring to mind that $\Omega \subseteq E$ is named
\begin{itemize}
\item[$\bullet$] closed if ${\rm cl}\, \Omega = \Omega$;
\item[$\bullet$] compact if every family of open sets whose union includes 
$\Omega$ contains a finite number of sets whose union includes $\Omega$;
\item[$\bullet$] topologically solid if ${\rm int}\, \Omega \neq \emptyset$;
\item[$\bullet$] algebraically solid if ${\rm cor}\, \Omega \neq \emptyset$.
\end{itemize}

Likewise, for any set $\Omega \subseteq E^*$ we will use all the topological notions (based on the dual norm topology in $E^*$) and algebraic notions introduced above.

\begin{remark}[Strong and weak topologies] \label{rem:StrongAndWeak2}
Let $(E, ||\cdot||)$ be a real normed space. 
Given any set $\Omega \subseteq E$, the weak closure of $\Omega$, that is denoted by ${\rm cl}_{w}\, \Omega$, is the closure of $\Omega$ w.r.t. the weak topology $\tau_w$.
A set $\Omega \subseteq E$ is called weakly closed if ${\rm cl}_{w}\, \Omega = \Omega$; weakly compact if $\Omega$ is $\tau_w$-compact. 
It is well known that ${\rm cl}\, \Omega \subseteq {\rm cl}_{w}\, \Omega$, hence the weak closedness of $\Omega$ implies the closedness of $\Omega$. 
Moreover, it is known that the weak closure contains all weakly sequential limit points of $\Omega$, i.e.,
$$\{x \in E \mid \exists\, (x_n) \subseteq \Omega:\; x_n \rightharpoonup x \mbox{ for } n \to \infty\} \subseteq {\rm cl}_{w}\, \Omega,$$
where $x_n \rightharpoonup x$ means that $(x_n)$ converges to $x$ for $n \to \infty$ w.r.t. the weak topology $\tau_w$.
\end{remark}

\subsection{Convex Sets, Cones and their Bases} \label{subsec:convex_sets_norms_snorms}

As usual, we call a set $\Omega \subseteq E$ convex if $(1-\lambda) x + \lambda \overline{x} \in \Omega$ for all $x, \bar x \in \Omega$ and $\lambda \in (0,1)$. The smallest convex set of $E$ containing $\Omega$ (i.e., the convex hull of $\Omega$) is denoted by ${\rm conv}\,\Omega$. It is easy to check that $\Omega$ is convex if and only if ${\rm conv}\,\Omega = \Omega$.
Moreover, if $\Omega$ is convex and either ${\rm int}\, \Omega \neq \emptyset$ or $E$ is a Banach space and $\Omega$ is closed, then ${\rm cor}\, \Omega = {\rm int}\, \Omega$ (see Barbu and Precupanu \cite[Cor. 1.20, Rem. 1.24]{BarbuPrecupanu12}). A special convex and closed set in the real normed space $E$ is given by the unit ball of $E$ w.r.t. the norm $||\cdot||$, which is denoted by $$\mathbb{B} := \{x \in E \mid ||x|| \leq 1\}.$$ 
Obviously, we have $\mathbb{B} = {\rm conv}(\mathbb{S})$,
where
$$\mathbb{S} := \{x \in E \mid ||x|| = 1\}$$
is the unit sphere of $E$ w.r.t. the norm $||\cdot||$.

\begin{remark}[Strong and weak topologies] \label{rem:StrongAndWeak3} The following facts are known:
\begin{itemize}
    \item[$\bullet$] According to Jahn \cite[Th. 3.24]{Jahn2011}, for any nonempty, convex subset $\Omega$ of a real
normed space $E$, the set $\Omega$ is closed if and only if it is weakly closed,  and moreover, ${\rm cl}_w\, \Omega = {\rm cl}\, \Omega$. 
\item[$\bullet$] Given any reflexive Banach space $(E, ||\cdot||)$, it is known that the unit ball $\mathbb{B}$ is a nonempty, convex, weakly compact set in $E$. Any weakly closed subset of this unit ball (e.g. if this subset is closed and convex) is weakly compact as well. 
\item[$\bullet$] In any infinite dimensional real normed space $(E, ||\cdot||)$, we have
$$\mathbb{S} \neq \mathbb{B} = {\rm cl}_w\,\mathbb{S}$$
(see, e.g., Kesavan \cite[Ex. 5.1.2]{Kesavan23}),
hence $||\cdot||$ is not weakly continuous (i.e., not continuous w.r.t. the weak topology $\tau_w$). However, $||\cdot||$ is continuous w.r.t the norm topology.
\end{itemize}
\end{remark}

\bigskip

\noindent A set $K \subseteq E$ is called a cone if $0 \in K = \mathbb{R}_+ \cdot K$. A cone $K \subseteq E$ is said to be
\begin{itemize}
	\item[$\bullet$] convex if $K + K = K$ (or equivalently, $K$ is a convex set); 
	\item[$\bullet$] nontrivial if $\{0\} \neq K \neq E$;
	\item[$\bullet$] pointed if $\ell(K) : = K \cap (-K) = \{0\}$. 
\end{itemize}
Given a convex cone $K \subseteq E$, the set $\ell(K)$ is known as the lineality space of $K$.  In this case, $K$ is a linear subspace of $E$ if and only if $K = \ell(K)$.
The topological dual cone of a cone $K \subseteq E$ is given by
$$
K^+ := \{y^* \in E^* \mid \forall\, k \in K:\; y^*(k) \geq 0\}.
$$
Furthermore, the subset
$$
K^\# := \{y^* \in E^* \mid \forall\, k \in K \setminus \{0\}:\; y^*(k) > 0\}
$$ 
of $K^+$ is of special interest.
It is important to mention that both sets $K^+$ and $K^\#$ are convex for any (not necessarily convex) cone $K \subseteq E$. Moreover, one has 
\begin{equation}
\label{eq:K+=convK+=clvonvK+}
  K^{+} = ({\rm conv}\,K)^+ = ({\rm cl}({\rm conv}\,K))^+ \quad \text{and} \quad ({\rm cl}({\rm conv}\,K))^\# \subseteq K^{\#} = ({\rm conv}\,K)^\#,   
\end{equation}
but the inclusion $({\rm cl}({\rm conv}\,K))^\# \subseteq K^{\#}$ can be strict (see G\"opfert et al.  \cite[p. 55]{GoeRiaTamZal2023}).

In the following definition, we recall a general base concept for cones in real linear topological spaces (cf. G\"opfert et al.  \cite[Def. 2.1.42]{GoeRiaTamZal2023}). 

\begin{definition} \label{def:baseKnowngeneral1}
	Consider a nontrivial cone $K \subseteq E$. A set $B \subseteq K$ is called a base for $K$, if
	$B$ is a nonempty set,  and $K = \mathbb{R}_+ \cdot B$ with $0 \notin {\rm cl}\, B$.
    Moreover, $K$ is said to be well-based if there exists a bounded, convex base.
\end{definition}

\begin{remark}
Taking into account \eqref{eq:K+=convK+=clvonvK+}, for any nontrivial cone $K \subseteq E$, it is well-known that
\begin{itemize}
    \item \cite[Prop. 2.2.23]{GoeRiaTamZal2023}:
    $${\rm conv}\,K \text{ has a convex base} \iff  K^\# \neq \emptyset ,$$
    \item \cite[Prop. 2.2.32]{GoeRiaTamZal2023}, \cite{Jam1970}:
    \begin{equation} \label{eq:wellbasedIntK+neqEmpty}
        {\rm conv}\,K \text{ is well-based}\iff {\rm cl}({\rm conv}\,K) \text{ is well-based} \iff {\rm int}\,K^{+}\neq\emptyset .
    \end{equation} 
\end{itemize}
\end{remark}

Moreover, we recall another (algebraic) base concept for cones in linear spaces (cf. G\"opfert et al. \cite[Def. 2.1.14]{GoeRiaTamZal2023}, Jahn \cite[Def. 1.10 (d)]{Jahn2011}).

\begin{definition} \label{def:baseKnowngeneral}
	Consider a nontrivial cone $K \subseteq E$. A set $B \subseteq K \setminus \{0\}$ is called an algebraic base for $K$, if
	$B$ is a nonempty set,  and every $x \in K \setminus \{0\}$ has a unique representation of the form 
	$x = \lambda b$ for some $\lambda > 0$ and some $b \in B$.
\end{definition}

\begin{remark}
If $B$ is base in the sense of Definition \ref{def:baseKnowngeneral1} or Definition \ref{def:baseKnowngeneral} for the nontrivial cone $K \subseteq E$, then $K = \mathbb{R}_+ \cdot B$ and $K \setminus \{0\} = \mathbb{P} \cdot B$.
\end{remark}

Our focus will be on a specific type of base in normed spaces, which we will recall in the following definition.

\begin{definition} \label{def:normbase}
	A set $B \subseteq K$ given by
	$$
	B := B_{K} := \{x \in K \mid ||x|| = 1\} = K \cap \mathbb{S}
	$$
	is called the
	norm-base for the nontrivial cone $K \subseteq E$.
\end{definition}

Throughout the paper, for any nontrivial cone $C \subseteq E$, we define
$$
S_C := {\rm conv}(B_C) \quad \text{and} \quad S_C^0 := {\rm conv}(\{0\} \cup B_C).
$$

\begin{remark} \label{rem:baseSK}
Consider a nontrivial cone $K \subseteq E$. Notice that the norm-base $B_K$ is a base for $K$ in the sense of both Definitions \ref{def:baseKnowngeneral1} and \ref{def:baseKnowngeneral}.
It is easy to check that 
\begin{align*}
 {\rm conv}\, K &= \mathbb{R}_+ \cdot S_{{\rm conv}\, K} = \mathbb{R}_+ \cdot S_{K},\\
 {\rm cl}({\rm conv}\, K) & = \mathbb{R}_+ \cdot S_{{\rm cl}({\rm conv}\, K)} = \mathbb{R}_+ \cdot {\rm cl}\, S_{{\rm cl}({\rm conv}\, K)} = {\rm cl}(\mathbb{R}_+ \cdot {\rm cl}\, S_{K}).
\end{align*}
Furthermore, if $0 \notin {\rm cl}\,S_K$, then $\mathbb{R}_+ \cdot {\rm cl}\, S_{K}$ is closed by the boundedness of ${\rm cl}\, S_{K} \, (\subseteq \mathbb{B})$ (see Göpfert et al. \cite[Lem. 2.1.43 (iii)]{GoeRiaTamZal2023}), and so 
$$
{\rm cl}({\rm conv}\, K) = \mathbb{R}_+ \cdot {\rm cl}\,S_{K}.
$$
Thus, we get the following equivalences:
    \begin{align*}
    & 0 \notin {\rm cl}\,S_K\\
      \Longleftrightarrow\quad  & S_K \text{ is a bounded, convex base (in the sense of Definition \ref{def:baseKnowngeneral1}) for } {\rm conv}\, K\\
   \Longleftrightarrow \quad &  {\rm cl}\,S_K \text{ is a closed, bounded, convex base (in the sense of Definition \ref{def:baseKnowngeneral1}) for } {\rm cl}({\rm conv}\, K).
    \end{align*}
Taking into account \eqref{eq:wellbasedIntK+neqEmpty} we have 
\begin{equation} \label{eq:0NotInClSKIntK+NeqEmpty}
0 \notin {\rm cl}\,S_K \Longrightarrow  {\rm conv}\, K \text{ (respectively, } {\rm cl}({\rm conv}\, K) \text{) is well-based} \iff {\rm int}\, K^+ \neq \emptyset.
\end{equation}
A more detailed analysis of the condition $0 \notin {\rm cl}\,S_K$ is given in the Theorems \ref{th:0notinclSK} and \ref{th:solidness_augmented_dual_cones}.
\end{remark}

Given a nontrivial cone $K \subseteq E$ with the norm-base $B_{K}$, we consider the set 
$$
K^{w\#} := \{x^* \in E^* \mid \forall\, x \in {\rm cl}_w\,B_K:\;  x^*(x) > 0\} \;\; (\subseteq K^\#).
$$

In the following theorem, we derive useful representations for the algebraic interior of the topological dual cone $K^{+}$ of a nontrivial (not necessarily convex) cone $K \subseteq E$. 

\begin{theorem} \label{th:properties_top_dual_cones}
	Assume that $E$ is a real normed space, $K \subseteq E$ is a nontrivial cone, and $B_{K}$ is the norm-base of $K$. Then, the following assertions are valid:
    \begin{itemize}
    	\item[$1^\circ$] $\{x^* \in E^* \mid \inf_{x \in B_K}\, x^*(x) > 0\} \subseteq {\rm cor}\, K^{+} \subseteq K^\# \subseteq K^+ = \{x^* \in E^* \mid \inf_{x \in B_K}\, x^*(x) \geq 0\}$.
        \item[$2^\circ$] If ${\rm cl}_w\,B_K$ is weakly compact (e.g. if $E$ is also reflexive), then ${\rm cor}\, K^{+} \supseteq K^{w\#}$.
        \item[$3^\circ$] If $0 \notin {\rm cl}_w\,B_K$, then ${\rm cor}\, K^{+} \subseteq K^{w\#}$.
        \item[$4^\circ$] If $B_K$ is weakly compact, or $K$ is weakly closed (e.g. if $K$ is closed and convex) and ${\rm cl}_w\,B_K$ is weakly compact with $0 \notin {\rm cl}_w\,B_K$, then ${\rm cor}\, K^{+}  = K^{w\#} = K^\#$.
        \item[$5^\circ$] If $E$ is also reflexive and $0 \notin {\rm cl}_w\,B_{{\rm cl}({\rm conv}\, K)}$ (e.g. if $E$ has finite dimension), then\\ ${\rm cor}\, K^{+}  = ({\rm cl}({\rm conv}\, K))^{w\#} = ({\rm cl}({\rm conv}\, K))^\#$.
    \end{itemize}
\end{theorem}

\begin{proof} {\color{white}.}
    \begin{itemize}
	\item[$1^\circ$] 
    It is known that the inclusion ${\rm cor}\, K^{+} \subseteq K^{\#}$ is true (for any nontrivial cone $K \subseteq E$
    with $E$ a real topological linear space for which $E^*$ separates the elements of $E$). 
    Let us prove that $\{x^* \in E^* \mid \inf_{x \in B_K}\, x^*(x) > 0\} \subseteq {\rm cor}\, K^{+}$. Consider $x^* \in E^*$ such that $\alpha := \inf_{x \in B_K}\, x^*(x) > 0$ and take $y^* \in E^* \setminus \{0\}$. Then,
    $$(x^* + \varepsilon y^*)(x) \geq \alpha + \varepsilon y^*(x) \geq \alpha - |\varepsilon| ||y^*||_* > 0 \quad \text{for all } x \in B_K \text{ and }|\varepsilon|  <  \alpha ||y^*||_*^{-1}.$$ Hence, we conclude
    $x^* \in {\rm cor}\, K^{+}$. 
    The remaining relations follow immediately from the definitions of $K^\#$ and $K^+$.

	\item[$2^\circ$] 
	Take some $x^* \in K^{w\#}$ (i.e., $x^*(x) > 0$ for all $x \in {\rm cl}_w\,B_K$)  and an arbitrary $y^* \in E^*$.
	We are going to show that 
	\begin{equation}
	    \label{eq:corK+CorProperty}
	    x^* + [0, \varepsilon] \cdot y^* \subseteq  K^{+} \quad \mbox{for some }\varepsilon > 0.
	\end{equation}
	Since ${\rm cl}_w\,B_K$ is weakly compact and $x^*$ and $y^*$ are continuous, there are $\delta, C > 0$ such that $x^*(x) \geq \delta > 0$ and $|y^*(x)| \leq C$ for all $x \in {\rm cl}_w\,B_K$. Define $\varepsilon := \frac{\delta}{C} \, (> 0)$. Then, for any $x \in {\rm cl}_w\,B_K \supseteq B_K$ and $\bar \varepsilon \in [0, \varepsilon]$, we get
	$$
	(x^* + \bar \varepsilon  y^*)(x) = x^*(x) + \bar \varepsilon  y^*(x) \geq \delta + \bar \varepsilon  y^*(x)  \geq \delta - \bar \varepsilon  C \geq 0. 
	$$
    We conclude that \eqref{eq:corK+CorProperty} is valid; hence $x^* \in {\rm cor}\, K^{+}$.
	
	\item[$3^\circ$] Since $K^{+} = ({\rm cl}_w\,K)^{+}$ we have ${\rm cor}\, K^{+} = {\rm cor}\,({\rm cl}_w\,K)^+$.
	Take some $x^* \in {\rm cor}\, ({\rm cl}_w\,K)^+$. On the contrary assume that $x^* \notin K^{w\#}$, hence there is $k \in {\rm cl}_w\,B_K$ with $x^*(k) \leq 0$. By our assumption $0 \notin {\rm cl}_w\,B_K$ we also get that $k \in {\rm cl}_w\,B_K \subseteq ({\rm cl}_w\, K) \setminus \{0\}$. Since the topological dual space $E^*$ separates elements in $E$, there is $y^* \in E^*$ with $y^*(k) < y^*(0) = 0$. Then, we conclude that
	$$
	(x^* + \varepsilon y^*)(k) = x^*(k) + \varepsilon y^*(k) < 0 \quad \mbox{ for all } \varepsilon > 0,
	$$
	a contradiction to $x^* \in{\rm cor}\,({\rm cl}_w\,K)^+$.
	
	\item[$4^\circ$] If $B_K$ is weakly compact, the conclusion follows by assertion $1^\circ$ and the fact that $K^{w\#} = K^\#$.
    Assume that $K$ is weakly closed and ${\rm cl}_w\,B_K$ is weakly compact with $0 \notin {\rm cl}_w\,B_K$. Then, assertions $2^\circ$ and $3^\circ$ yield ${\rm cor}\, K^{+} = K^{w\#}$. Since    
    $B_K \subseteq {\rm cl}_w\,B_K \subseteq ({\rm cl}_w\,K) \setminus \{0\} = K \setminus \{0\}$, we get
	$K \setminus \{0\} = \mathbb{P}\cdot B_K \subseteq \mathbb{P}\cdot {\rm cl}_w\,B_K \subseteq K \setminus \{0\}$, hence $K^{w\#} = K^\#$.
    \item[$5^\circ$] This assertion follows immediately by $4^\circ$, since ${\rm cl}_w\,B_{{\rm cl}({\rm conv}\, K)}$ is weakly compact as a weakly closed subset of the weakly compact unit ball $\mathbb{B}$. 
    
    \end{itemize}	
\end{proof}

\begin{remark} \label{rem:ResGRTZ}
Notice, for any nontrival cone $K \subseteq E$, we have
\begin{equation}
\label{eq:int=corK+}
    {\rm int}\, K^{+} = {\rm cor}\, K^{+}
\end{equation}
since the dual cone $K^{+}$ of $K$ is a closed and convex set in the real Banach space $E^*$ (taking into account that $E$ is a real normed space). Moreover, if $K$ is a nontrival, closed, convex cone in the real normed space $E$, then the following facts are known:
\begin{itemize}
\item[$\bullet$] $K^\#$ is nonempty if $E$ is separable and $K$ is pointed (Krein–Rutman theorem; see Göpfert et al. \cite[p. 66]{GoeRiaTamZal2023}, Jahn \cite[Th. 3.38]{Jahn2011}).
\item[$\bullet$] $K^{\#}$ represents the quasi-interior of the dual cone $K^{+}$ w.r.t. the weak$^*$ topology $\tau_{w^*}$ (see {Bo{\c{t}}}, Grad and Wanka \cite[Prop. 2.1.1]{BotGradWanka}).
Furthermore, $K^{\#}$ represents the quasi-interior of the dual cone $K^{+}$ w.r.t. the norm topology in $E^*$ (see Göpfert et al. \cite[Prop. 2.2.25]{GoeRiaTamZal2023}).
\item[$\bullet$] $K^{\#}$ represents the interior of the dual cone $K^{+}$ w.r.t. $\tau_{w^*}$ if this interior is nonempty. Moreover, $K^{\#} = {\rm int}\, K^{+}$ if $E$ is a reflexive Banach space and this interior is nonempty (see Göpfert et al. \cite[Prop. 2.2.25]{GoeRiaTamZal2023}, Jahn \cite[Lem. 3.21]{Jahn2011}) or if $B_K$ is weakly compact (see G\"unther, Khazayel and Tammer \cite[Th. 2.18]{GuenKhaTam22}).
\end{itemize}

In Theorem \ref{th:properties_top_dual_cones}, we characterize the algebraic interior ${\rm cor}\, K^+$  of $ K^+$ under the assumption that $K$ is a nontrivial cone in the real normed space $E$. Corresponding results are shown in \cite[Prop. 2.2.25]{GoeRiaTamZal2023}, see also the references in \cite[Section 2.2]{GoeRiaTamZal2023}. 
Taking into account \eqref{eq:K+=convK+=clvonvK+},  \eqref{eq:wellbasedIntK+neqEmpty} and \eqref{eq:int=corK+}, for any nontrivial cone $K \subseteq E$,
from \cite[Prop. 2.2.33]{GoeRiaTamZal2023} (see also \cite[Th. 3.6]{Chiang2012}) we obtain
\begin{equation} \label{eq:0NotInClSconvK}
{\rm cor}\, K^{+}\neq\emptyset \iff {\rm int}\, K^+ \neq \emptyset  \iff
{\rm cl}({\rm conv}\,K)\; \mbox{ is well-based } \Longleftrightarrow 0\notin {\rm cl}\,S_{{\rm cl}({\rm conv}\,K)},    
\end{equation}
while from \cite[Prop. 2.2.25]{GoeRiaTamZal2023}, 
if $E$ is also reflexive, we get 
\[{\rm int}\, K^+ \neq \emptyset \quad \Longrightarrow \quad {\rm cor}\, K^+ = {\rm int}\, K^+ =  ({\rm cl}({\rm conv}\,K))^\#. \]
In particular, assertions $4^\circ$ and $5^\circ$ in Theorem \ref{th:properties_top_dual_cones} are related to \cite[Prop. 2.2.25]{GoeRiaTamZal2023}, since 
$${\rm int}\, K^+ \neq \emptyset \iff  0\notin {\rm cl}\,S_{{\rm cl}({\rm conv}\,K)}  \Longrightarrow 0 \notin {\rm cl}_w\,B_{{\rm cl}({\rm conv}\, K)}$$ 
(see also the upcoming Theorem \ref{th:0notinclSK}). Notice that the assertions in \cite[Prop. 2.2.25, Cor. 2.2.30]{GoeRiaTamZal2023} are derived in a more general (Hausdorff locally convex space) setting.  
\end{remark}

\subsection{Classical Separation Theorems for Convex Sets / Cones} \label{sec:linear_separation}

We will employ classical separation arguments for convex sets in order to prove our main strict cone separation results in Section  \ref{sec:cone_separation_normed}. 

\begin{proposition}
	\label{prop:seperationCompactJahn2} Suppose that $E$ is a real normed space and $\Omega^1, \Omega^2 \subseteq E$ are nonempty, closed, convex sets, where one of the sets $\Omega^1$ and $\Omega^2$ is weakly compact. If $\Omega^1 \cap \Omega^2 = \emptyset$, then there is $x^* \in E^* \setminus \{0\}$ such that
     $$\sup_{\omega^1 \in \Omega^1}\, x^*(\omega^1)  < \inf_{\omega^2 \in \Omega^2}\, x^*(\omega^2).$$
\end{proposition}

\begin{proof}
The classical result for strict linear separation of two convex sets, as described in Rudin \cite[Th. 3.4 (b)]{Rudin1973}, can be applied to the space $(E, \tau_w)$, resulting in this proposition.
\end{proof}

For the separation of two convex cones by a hyperplane we have the following result (see also Jahn \cite[Th. 3.22]{Jahn2011}).

\begin{proposition} \label{prop:Jahn_sep_two_closed_cones}
	Suppose that $E$ is a real normed space, $K, A \subseteq E$ are closed, convex cones, and ${\rm cl}\,S_K$ is weakly compact with $0 \notin {\rm cl}\,S_K$. If $A \cap (-K) = \{0\}$, then there is $x^* \in E^* \setminus \{0\}$ such that 
		\begin{equation}
            \label{eq:sep_lin_cones}
		    x^*(a) \geq 0 > x^*(k) \quad \mbox{for all } a \in A \mbox{ and  }k \in -K \setminus \{0\}.
		\end{equation}
\end{proposition} 

\begin{proof} 
 Since $0 \notin {\rm cl}\,S_K$ we have $K = {\rm cl}({\rm conv}\,K) = \mathbb{R}_+ \cdot {\rm cl}\,S_K$ (by Remark \ref{rem:baseSK}), and so $-K \setminus \{0\} = \mathbb{P} \cdot (-{\rm cl}\,S_K)$. Thus, under the assumption $A \cap (-K \setminus \{0\}) = \emptyset$, it follows $A \cap (-{\rm cl}\,S_K) = \emptyset$. Since both sets $A$ and  $-{\rm cl}\,S_K$ are nonempty, closed and convex, $-{\rm cl}\,S_K$ is weakly compact, by Proposition \ref{prop:seperationCompactJahn2} one can strictly separate $A$ and $-{\rm cl}\,S_K$ by using $x^* \in E^* \setminus \{0\}$. From this, it is easy to derive \eqref{eq:sep_lin_cones}.     
\end{proof}

\section{Properties of Augmented Dual Cones} \label{sec:augmented_dual_cones}

Assume that $(E, ||\cdot||)$ is a real normed space  and $K \subseteq E$ is a nontrivial cone. 
Following the definition of so-called topological augmented dual cones by Kasimbeyli \cite{Kasimbeyli2010} (see also G\"unther, Khazayel and Tammer \cite[Sec. 3]{GuenKhaTam22} for some extensions), we define
$$
K^{a+} := \{(x^*, \alpha) \in K^+ \times \mathbb{R}_+ \mid \forall\, y \in K:\; x^*(y) - \alpha ||y|| \geq 0\},
$$
and its subset
\begin{align*}
K^{a\#} & := \{(x^*, \alpha) \in K^\# \times \mathbb{R}_+ \mid \forall\, y \in K \setminus \{0\}:\;  x^*(y) - \alpha ||y|| > 0\}.
\end{align*}

\begin{remark}  \label{rem:adCones_Bases}
Taking into account the fact that for $(x^*, \alpha) \in E^* \times \mathbb{R}_+$ we have $x^*(y) \geq x^*(y) - \alpha ||y||$ for all $y \in E$, we can also write $E^*$ (or $K^+$) instead of $K^+$ and $K^{\#}$, respectively, in the definitions of $K^{a+}$ and $K^{a\#}$.
It is easy to check that 
\begin{equation}
\label{eq:K+=convKa+=clvonvKa+}
    K^{a+} = ({\rm conv}\,K)^{a+} = ({\rm cl}({\rm conv}\,K))^{a+} \quad \text{and} \quad ({\rm cl}({\rm conv}\,K))^{a\#} \subseteq K^{a\#} = ({\rm conv}\,K)^{a\#}.
\end{equation}
Moreover, using the norm-base $B_K$ of $K$, we get
\begin{align*}
K^{a+} & = \{(x^*, \alpha) \in K^+ \times \mathbb{R}_+ \mid \forall\, y \in B_K:\; x^*(y) \geq  \alpha\},\\
K^{a\#} & = \{(x^*, \alpha) \in K^\# \times \mathbb{R}_+ \mid \forall\, y \in B_K:\;  x^*(y) > \alpha\}.
\end{align*}
\end{remark}

In our separation approach, the following subset of $K^{a\#}$ will be of special importance,
\begin{equation}
\label{eq:defKawSharp}
K^{aw\#} := \{(x^*, \alpha) \in K^\# \times \mathbb{R}_+ \mid \forall\, y \in {\rm cl}_w\,B_K:\;  x^*(y) > \alpha\}.
\end{equation}
\begin{remark}
Clearly, if $B_K$ is weakly closed, then $K^{a\#} = K^{aw\#}$. Moreover, for any $x^* \in E^*$, we have
$$
||x^*||_* \geq \sup_{x  \in B_K}\,  x^*(x) \geq \inf_{x  \in B_K}\,  x^*(x) = \inf_{x  \in {\rm cl}_w\, B_K}\,  x^*(x) = \inf_{x  \in S_K}\,  x^*(x) = \inf_{x  \in {\rm cl}\,S_K}\,  x^*(x) \geq -||x^*||_*.
$$
\end{remark}

Based on \cite[Lem. 3.7]{GuenKhaTam22}, the main properties of the topological augmented dual cone $K^{a+}$  and
its subsets  $K^{a\#}$ and $K^{aw\#}$ are as follows:

\begin{itemize}
	\item[$\bullet$] $K^\theta \times \{0\} = K^{a\theta} \cap (E^* \times \{0\}) \subseteq K^{a\theta}$ for all $\theta \in \{+, \#, w\#\}$.
	\item[$\bullet$] {Cone property:} $\mathbb{P} \cdot K^{a\theta} = K^{a\theta}$ for all $\theta \in \{+, \#, w\#\}$,
	and $(0,0) \in K^{a+}$ (i.e., $K^{a+}$ is a cone). 

	\item[$\bullet$] {Convexity:} $K^{a\theta}$ is convex (or equivalently, $K^{a\theta} + K^{a\theta} = K^{a\theta}$) for all $\theta \in \{+, \#, w\#\}$.
	
	\item[$\bullet$] {Pointedness:} $\ell(K^{a+}) = \ell(K^+) \times \{0\}$, hence $K^{a+}$ is pointed $\iff$ $K^+$ is pointed.
	Moreover, 
	$K^{a+} \setminus \ell(K^{a+}) = K^{a+} \cap (K^+ \setminus \ell(K^+) \times \mathbb{P})$.
		
	\item[$\bullet$] {Nontriviality:} $K^{a+}$ is nontrivial $\iff$ $K^+ \neq \{0\}$.\\
\end{itemize}

The following lemma provides conditions for the existence of elements in the sets $K^{a+} \cap (E^* \times \mathbb{P})$, $K^{aw\#}$ and $K^{aw\#} \cap (E^* \times \mathbb{P})$, respectively.

\begin{lemma} \label{lem:KaPlushasPosElement}
Let $E$ be a real normed space and $K$ a nontrivial cone with the norm-base $B_{K}$.  Then, the following assertions are valid:
	\begin{itemize} 
		\item[$1^\circ$] If ${\rm cl}_w\,B_K$ is weakly compact  (e.g. if $E$ is also reflexive), then $(x^*, c) \in K^{a+} \cap (E^* \times \mathbb{P})$ for $x^* \in K^{w\#}$ and $c := {\rm min}_{x \in {\rm cl}_w\,B_K}\, x^*(x)$.
        \item[$2^\circ$]
        If $K$ is weakly closed and ${\rm cl}_w\,B_K$ is weakly compact with $0 \notin {\rm cl}_w\,B_{K}$, then $(x^*, c) \in K^{a+} \cap (E^* \times \mathbb{P})$ for $x^* \in K^{\#}$ and $c := {\rm min}_{x \in {\rm cl}_w\,B_K}\, x^*(x)$.
		\item[$3^\circ$]
		If $(x^*, \alpha) \in  K^{a+} \cap (E^* \times \mathbb{P})$, then $(x^*, \alpha -\varepsilon) \in K^{aw\#}$ for all $\varepsilon \in (0, \alpha]$.
	\end{itemize}	
\end{lemma}

\begin{proof}
    \begin{itemize} 
        	\item[$1^\circ$] Since  $x^* \in K^{w\#}$,  we infer $x^*(x) > 0$ for all $x \in {\rm cl}_w\,B_{K}$. By the continuity of $x^*$ and the weak compactness of ${\rm cl}_w\,B_{K}$, we conclude $c = {\rm min}_{x \in {\rm cl}_w\,B_{K}} x^*(x) > 0$. 
        	This shows that $(x^*, c) \in K^{a+} \cap (E^* \times \mathbb{P})$.
            \item[$2^\circ$] In view of Theorem \ref{th:properties_top_dual_cones} ($4^\circ$), if $K$ is weakly closed and ${\rm cl}_w\,B_K$ is weakly compact with $0 \notin {\rm cl}_w\,B_{K}$, then $K^\# = K^{w\#}$. Thus, the conclusion follows by assertion $1^\circ$. 
        	\item[$3^\circ$] Take some $(x^*, \alpha) \in  K^{a+}$ with $\alpha > 0$, hence $x^*(x) \geq \alpha > 0$ for all $x \in {\rm cl}({\rm conv}(B_K)) \supseteq {\rm cl}_w\,B_{K}$. Thus, for any $\varepsilon \in (0, \alpha]$, we have $x^*(x) > \alpha - \varepsilon \geq 0 $ for all $x \in {\rm cl}_w\,B_{K}$, and so $(x^*, \alpha -\varepsilon) \in  K^{aw\#}$.	
     \end{itemize}   	
\end{proof}

Next, we will study the nonemptyness of the sets $K^{a+} \cap (E^* \times \mathbb{P})$, $K^{a\#} \cap (E^* \times \mathbb{P})$, $K^{aw\#}$, and $K^{aw\#} \cap (E^* \times \mathbb{P})$, and we will relate these conditions to the condition $0 \notin {\rm cl}\,S_K$.

\begin{theorem} \label{th:0notinclSK}
Let $E$ be a real normed space, and let $K$ be a nontrivial cone with the norm-base $B_{K}$. 
Then, the following assertions hold:
\begin{itemize}
    \item[$1^\circ$] $0 \notin S_{K}  \Longrightarrow K$ is pointed.
    \item[$2^\circ$] 
    Assume that $K$ is convex.\\    
    a) $S_K^0 = K \cap \mathbb{B}$, and if $K$ is closed, then ${\rm cl}\,S_K^0 = K \cap \mathbb{B}$. \\
    b) $0 \in S_K \iff S_K = S_K^0$; and moreover, $0 \in {\rm cl}\,S_K \iff {\rm cl}\,S_K = {\rm cl}\,S_K^0$.
    \item[$3^\circ$] $
    K^{aw\#} \cap (E^* \times \mathbb{P})\neq \emptyset  \iff K^{a\#} \cap (E^* \times \mathbb{P})\neq \emptyset \iff K^{a+} \cap (E^* \times \mathbb{P}) \neq \emptyset \\
    \iff 
    {\rm cor}\, K^{+} \neq \emptyset \iff 
    {\rm int}\, K^{+} \neq \emptyset \\
    \iff 0 \notin {\rm cl}\,S_K$.
    \item[$4^\circ$] If ${\rm cl}_w\,B_K$ is weakly compact  (e.g. if $E$ is also reflexive), then \\
    $K^{w\#} \neq \emptyset 
    \iff  K^{aw\#}  \neq \emptyset \iff 0 \notin {\rm cl}\,S_K$.
    \end{itemize}
\end{theorem}

\begin{proof}
\begin{itemize}
    \item[$1^\circ$]  Assuming that $K$ is not pointed, i.e., $0 \neq x \in \ell(K) = K \cap (-K)$, we find $0 \neq \bar x \in B_K \cap \ell(K)$, hence $-\bar x \in B_K \cap \ell(K)$. Consequently, $0 \in {\rm conv}(B_{K}) = S_{K}$. 

    \item[$2^\circ$]  a) The inclusion $S_{K}^0 \subseteq K \cap \mathbb{B}$ is obvious. Take $x \in K \cap \mathbb{B}$ and put $\lambda := ||x|| \;(\leq 1)$. If $\lambda \in \{0,1\}$, then $x \in B_{K} \cup \{0\} \subseteq S_{K}^0$; if $\lambda \in (0,1)$, then $x = (1-\lambda) \cdot 0 + \lambda \cdot x'$ for $0, x' := \lambda^{-1} x \in B_{K} \cup \{0\} \subseteq S_{K}^0$, hence $x \in S_{K}^0$ by the convexity of $S_{K}^0$.
    
    Since $S_K^0 = K \cap \mathbb{B}$ and $K$ is closed, we know that $K \cap \mathbb{B} = S_K^0 = {\rm cl}\,S_K^0$

    b) Obviously, $S_K \subseteq S_K^0$, and if $0 \in S_K$, then $B_K \cup \{0\} \subseteq S_K$, and so $S_K^0 \subseteq S_K$.  Analogously, ${\rm cl}\,S_K \subseteq {\rm cl}\,S_K^0$, and if $0 \in {\rm cl}\,S_K$, then $B_K \cup \{0\} \subseteq {\rm cl}\,S_K$, and so ${\rm cl}\,S_K^0 \subseteq {\rm cl}\,S_K$.
    
	\item[$3^\circ$] The equivalences $K^{a\#} \cap (E^* \times \mathbb{P})\neq \emptyset \iff K^{a+} \cap (E^* \times \mathbb{P}) \neq \emptyset \iff 0 \notin {\rm cl}\,S_K$ follow from  G\"unther, Khazayel and Tammer \cite[Lem. 3.11 (3)]{GuenKhaTam22}. By Lemma \ref{lem:KaPlushasPosElement} ($3^\circ$)  and the fact that $K^{aw\#} \subseteq  K^{a+}$, we conclude $K^{a+} \cap (E^* \times \mathbb{P}) \neq \emptyset \iff K^{aw\#} \cap (E^* \times \mathbb{P}) \neq \emptyset$. Furthermore, as observed in \eqref{eq:0NotInClSKIntK+NeqEmpty} in Remark \ref{rem:baseSK}, 
    $0 \notin {\rm cl}\,S_K \Longrightarrow {\rm int}\, K^{+} \neq \emptyset$, while \eqref{eq:int=corK+} and \eqref{eq:0NotInClSconvK} in Remark \ref{rem:ResGRTZ} show that
    ${\rm int}\, K^{+} \neq \emptyset
    \iff {\rm cor}\, K^{+} \neq \emptyset
    \iff 0 \notin {\rm cl}\,S_{{\rm cl}({\rm conv}\,K)} \Longrightarrow 0 \notin {\rm cl}\,S_K$.
	\item[$4^\circ$] 
	Obviously, $K^{aw\#} \cap (E^* \times \mathbb{P})\neq \emptyset\Longrightarrow K^{aw\#}  \neq \emptyset \Longrightarrow K^{w\#} \neq \emptyset$, while $K^{w\#}  \neq \emptyset \Longrightarrow K^{a+} \cap (E^* \times \mathbb{P})\neq \emptyset$ follows by Lemma \ref{lem:KaPlushasPosElement} ($1^\circ$).  The conclusion is reached by applying $3^\circ$.
\end{itemize}
\end{proof}

\begin{remark}
    Consider a nontrivial cone $K \subseteq E$ in the real normed space $E$. By Theorem \ref{th:0notinclSK} ($3^\circ$) we get some new insight in some results by Eichfelder and Kasimbeyli \cite[Th. 3.13 and Prop. 3.14]{EichfelderKasimbeyli2014} since
    $${\rm conv}\, K \text{ is well-based } \iff 0 \notin {\rm cl}\,S_K \iff K^{a+} \cap (E^* \times \mathbb{P})\neq \emptyset \iff K^{a\#} \cap (E^* \times \mathbb{P})\neq \emptyset$$
    and
    $$K^{a+} \cap (E^* \times \mathbb{P})\neq \emptyset \iff 0 \notin {\rm cl}\,S_K \iff {\rm cl}\,S_K \text{ is a closed, bounded, convex base for } {\rm cl}({\rm conv}\, K).$$
\end{remark}

The implication $K^{a\#}  \neq \emptyset \Longrightarrow 0 \notin {\rm cl}\,S_K$ is in general not true, as the next example with $(x^*, 0) \in K^{a\#}$ and $0 \in {\rm cl}\,S_K$ shows.

\begin{example} \label{ex:counter_example_0inclSK}
Consider the sequence space $E := \ell_2$ endowed with its usual norm $||\cdot||_{\ell_2}: \ell_2 \to \mathbb{R}$, defined by
$$
||x||_{\ell_2} := \sqrt{\sum_{i = 1}^\infty x_i^2} \quad \mbox{for all }x = (x_1, x_2, \ldots) \in \ell_2.
$$
Notice that ($\ell_2$, $||\cdot||_{\ell_2}$) is a Hilbert space (and so also a reflexive Banach space).
Moreover, consider the natural ordering cone in $\ell_2$ (which is a nontrivial, pointed, closed, convex cone),
\begin{equation}
\label{eq:standard_cone_l2}
    K := \{x = (x_1, x_2, \ldots) \in \ell_2 \mid \forall i \in \mathbb{N}: \; x_i \geq 0 \}.
\end{equation}
Then, $(x^*, 0) \in K^{a\#}$ with $x^* := (\frac{1}{n})_{n \in \mathbb{N}} \in K^\#$. Considering the well-known facts $K = K^+$ and ${\rm int}\, K = {\rm cor}\, K = \emptyset$ (cf. Jahn \cite[Ex. 1.48]{Jahn2011}), the condition $0 \in {\rm cl}\,S_K$ follows from Theorem \ref{th:0notinclSK} ($3^\circ$). In fact, it is easy to check that $0$ is a weakly sequential limit point of $B_K$, hence $0 \in {\rm cl}_w\, B_K \subseteq {\rm cl}_w\, S_K = {\rm cl}\, S_K$. Since $0 \in ({\rm cl}_w\, B_K) \setminus B_K$, the norm base $B_K$ is not weakly compact, unlike ${\rm cl}\, S_K$\; ($= S_K^0 = K \cap \mathbb{B}$ by Theorem \ref{th:0notinclSK} ($2^\circ$)), which is weakly compact.  
\end{example}

\begin{remark}
Taking into account Theorem \ref{th:0notinclSK} ($4^\circ$), our Example \ref{ex:counter_example_0inclSK} shows that the case ${\rm cor}\, K^{+}= K^{w\#} = \emptyset \neq K^{\#}$ may happen for a nontrivial, closed, pointed, convex cone $K$ in a real reflexive Banach space $E$. 
\end{remark}

Below is a simpler example (with $K^{a\#}  \neq \emptyset$ and $0 \in {\rm cl}\,S_K$) in any real normed space (of at least dimension two) with a nontrivial, convex cone $K$ that is not closed.

\begin{example} \label{ex:counter_example2_0inclSK}
For any $x^* \in E^* \setminus \{0\}$, we consider the nontrivial, convex cone 
$$K := \{0\} \cup \{x \in E \mid x^*(x) > 0\}.$$ 
Then, $K^{a\#} = \mathbb{P} \cdot \{(x^*, 0)\} \neq \emptyset$ and $B_K = \{x \in \mathbb{S} \mid x^*(x) > 0\}$. Moreover, if $E$ has dimension two or higher, then $S_K = \{x \in \mathbb{B} \mid x^*(x) > 0\}$ and $0 \in \{x \in \mathbb{B} \mid x^*(x) \geq 0\} = {\rm cl}\, S_K$.
\end{example}

In the next theorem, we derive useful representations of the algebraic interior of the topological augmented dual cone $K^{a+}$ (see also Theorem \ref{th:properties_top_dual_cones}).

\begin{theorem} \label{th:solidness_augmented_dual_cones}
	Assume that $E$ is a real normed space, and $K \subseteq E$ is a nontrivial cone with the norm-base $B_{K}$.
	Then, the following assertions hold:
    \begin{itemize}
    \item[$1^\circ$] 
        ${\rm cor}\, K^{a+} = \{(x^*, \alpha) \in K^+ \times \mathbb{P} \mid \inf_{x  \in B_K}\,  x^*(x) > \alpha \} \subseteq K^{a\#} \cap (E^*\times \mathbb{P})$.
    \item[$2^\circ$] 
        If ${\rm cl}_w\,B_K$ is weakly compact  (e.g. if $E$ is also reflexive), then\\
        a) ${\rm cor}\, K^{a+} = K^{aw\#} \cap (E^*\times \mathbb{P})$;\\ 
        b) ${\rm cor}\, K^{a+} \neq \emptyset \iff 0 \notin {\rm cl}\,S_K$.
    \item[$3^\circ$] 
        If $B_{K}$ is weakly compact, then ${\rm cor}\, K^{a+} =  K^{a\#} \cap (E^*\times \mathbb{P})$.	
    \end{itemize}
\end{theorem}

\begin{proof}
\begin{itemize}
	\item[$1^\circ$]
    For $(x^*, \alpha) \in {\rm cor}\, K^{a+}$, there exists $\delta > 0$ such that $(x^*, \alpha \pm \delta) = (x^*, \alpha) \pm \delta (0,-1) \in K^{a+}$; hence $\alpha - \delta \geq 0$, since $(x^*, \alpha - \delta ) \in K^{a+}$, and $\inf_{x  \in B_K}\,  x^*(x) \geq \alpha + \delta > \alpha$, since $(x^*, \alpha + \delta) \in K^{a+}$.
 
    Conversely, take $(x^*, \alpha)$ with $\inf_{x  \in B_K}\,  x^*(x) > \alpha > 0$ and $(y^*, \beta) \in E^* \times \mathbb{R}$. Put $\gamma := 1 + ||y^*||_*$. Then, $|y^*(y)| \leq ||y^*||_* < \gamma$ for $y \in \mathbb{S} \supseteq B_K$. Define the positive number 
    $$\delta := (\gamma + |\beta|)^{-1} \min\left\{\alpha\; , \left(\inf_{x  \in B_K}\,  x^*(x)\right) - \alpha\right\}.$$
    Consequently, for $|\lambda| \leq \delta$ and $y \in B_K$, one has
    $$
    \alpha + \lambda \beta \geq \alpha - |\lambda \beta| \geq \alpha - |\lambda| (\gamma + |\beta|) \geq \alpha - \delta (\gamma + |\beta|) \geq 0
    $$
    and
    \begin{align*}
        x^*(y) + \lambda y^*(y) - (\alpha + \lambda \beta) & \geq  \left(\inf_{x  \in B_K}\,  x^*(x)\right) - \alpha - \gamma |\lambda| - |\lambda \beta| \\
        &\geq \left(\inf_{x  \in B_K}\,  x^*(x)\right) - \alpha - \delta (\gamma + |\beta|) \\
        & \geq 0,
    \end{align*}      
    whence $(x^*, \alpha) + \lambda (y^*, \beta) \in K^{a+}$. It follows that $(x^*, \alpha) \in {\rm cor}\, K^{a+}$.
    
    Finally, the inclusion ${\rm cor}\, K^{a+} \subseteq K^{a\#} \cap (E^*\times \mathbb{P})$ is obvious.
  
    \item[$2^\circ$] The equality in a) follows by $1^\circ$, since $\inf_{x  \in {\rm cl}_w\, B_K} x^*(x)$ ($= \inf_{x  \in B_K} x^*(x)$) is attained, and thus the equivalence in b) follows by Theorem \ref{th:0notinclSK} ($3^\circ$).
 
    \item[$3^\circ$] This assertion is an immediate consequence of $1^\circ$, since $\inf_{x  \in B_K} x^*(x)$ is attained. 
    \end{itemize}
\end{proof}

We conclude the Section \ref{sec:augmented_dual_cones} by applying our results given in Theorems \ref{th:properties_top_dual_cones} and \ref{th:0notinclSK}  to the class of Bishop-Phelps cones, which will play an important role in our Sections \ref{sec:nonSepAppAndBPCones} and \ref{sec:cone_separation_normed}.

\begin{example}
    Assume that $(E, ||\cdot||)$ is a real reflexive Banach space. For a given $y^* \in E^*$ with $||y^*||_* > 1$, consider
    the Bishop-Phelps cone 
    $$
    C := C(y^*) = \{ x \in E \mid y^*(x) \geq ||x||\}
    $$
    and its norm-base 
    $$
    B_C = \{x \in C \mid ||x|| = 1\} = \{ x \in \mathbb{S} \mid y^*(x) \geq 1\}.
    $$
    Define $D := \{ x \in \mathbb{B} \mid y^*(x) \geq 1\}$. It is easy to verify that $S_C \subseteq D$ (taking into account that $D$ is a convex set), and
    that $S_C = D$ if $E$ has dimension two or higher.
    Additionally, since $D$ is closed, we conclude that
    $$
    {\rm cl}_w\,B_C \subseteq {\rm cl}\,S_C \subseteq D.
    $$
    Now, observe that $0 \notin D$, hence $0 \notin {\rm cl}\,S_C$ and so $0 \notin {\rm cl}_w\,B_C$. By Theorem \ref{th:properties_top_dual_cones} ($5^\circ$) and Theorem \ref{th:0notinclSK} ($4^\circ$) we conclude that 
    $${\rm int}\, C^{+} = {\rm cor}\, C^{+} = C^{w\#} = C^{\#} \neq \emptyset.$$
\end{example}

\section{Nonlinear Separation Approach based on Bishop-Phelps Cones} \label{sec:nonSepAppAndBPCones}

In this section, we illustrate our nonlinear separation approach based on Bishop-Phelps cones. We begin by recalling some useful properties of the separating function $\varphi_{x^*, \alpha} : E \to \mathbb{R}$ with $(x^*, \alpha) \in E^* \times \mathbb{R}_+$, which we introduced in \eqref{f12} through
\begin{equation*}
\varphi_{x^*, \alpha}(y) = x^*(y) + \alpha ||y|| \quad \mbox{for all } y \in E.
\end{equation*}	
The following properties of $\varphi_{x^*, \alpha}$ are valid:
\begin{itemize}
	\item[$\bullet$] $\varphi_{x^*, \alpha}$ is continuous and sublinear (hence convex).
	\item[$\bullet$] The sublevel set w.r.t. the level 0 of $\varphi_{x^*, \alpha}$, namely the set
	$$C_{x^*, \alpha}^{\leq} := \{x \in E \mid \varphi_{x^*, \alpha}(x) \leq 0\},$$
	is a closed, convex cone. If $\alpha > 0$, then $C_{x^*, \alpha}^{\leq}$ is pointed, and if $||x^*||_* > \alpha$, then $C_{x^*, \alpha}^{\leq}$ is nontrivial and topologically solid.
    
	\item[$\bullet$]The strict sublevel set w.r.t. the level 0 of $\varphi_{x^*, \alpha}$, namely the set
	$$C_{x^*, \alpha}^{<} := \{x \in E \mid \varphi_{x^*, \alpha}(x) < 0\},$$
	is nonempty if and only if $||x^*||_* > \alpha$. Moreover, if $C_{x^*, \alpha}^{<}$ is nonempty  (e.g. if $(-x^*, \alpha) \in K^{a\#}$ for a nontrivial cone $K \subseteq E$), then 
	$
	{\rm int}\, C_{x^*, \alpha}^{\leq} = C_{x^*, \alpha}^{<} = \mathbb{P} \cdot C_{x^*, \alpha}^{<} = C_{x^*, \alpha}^{<} + C_{x^*, \alpha}^{<}.
	$
\end{itemize}

\begin{remark}
An interesting observation (see also Jahn \cite{Jahn2022, Jahn2023}) is that $C_{x^*, \alpha}^{\leq}$ (with $\alpha > 0$)
is actually a Bishop-Phelps cone $C(-\alpha^{-1} x^*)$, i.e.,
$$C_{x^*, \alpha}^{\leq} = \{x \in E \mid \varphi_{x^*, \alpha}(x) = x^*(x) + \alpha ||x|| \leq 0\} = C_{\alpha^{-1} x^*, 1}^{\leq} = - C(\alpha^{-1} x^*) = C(-\alpha^{-1} x^*).$$
Consequently, the above mentioned properties of $C_{x^*, \alpha}^{\leq}$ and $C_{x^*, \alpha}^{<}$ follow directly from the properties of $C(-\alpha^{-1} x^*)$ and $C_>(-\alpha^{-1} x^*)$.    
\end{remark}

Assume that $K$ and $A$ are nontrivial cones in the real normed space $E$. As mentioned in the Introduction, we will focus in our upcoming Section \ref{sec:cone_separation_normed} on the separation of two (not necessarily convex) cones by a (convex) cone $C \subseteq E$. We state conditions such that the cones $-K$ and $A$ are strictly separated by a convex cone of Bishop-Phelps type. More precisely, we like to find $(x^*, \alpha) \in K^{a\#} \cap (E^* \times \mathbb{P})$ such that $-K$ and $A$ are strictly separated by $C = C_{x^*, \alpha}^{\leq}$, which means in terms of functional analysis,
\begin{align}
	x^*(a) + \alpha ||a|| &  >  0 > x^*(k) + \alpha ||k||   \quad \mbox{for all } a \in A\setminus \{0\} \mbox{ and  }k \in -K \setminus \{0\}.  \label{sep_non_top_4}
\end{align}	

\begin{remark}
    In our paper, we follow the approach (by Kasimbeyli \cite{Kasimbeyli2010}) to separate $-K$ and $A$ instead of $K$ and $A$, and to consider a pair $(x^*, \alpha) \in K^{a+}$. The separation approach for $-K$ and $A$ is known to be useful in applications from the fields of vector optimization and order theory (see, e.g., Jahn \cite{Jahn2011} and Kasimbeyli \cite{Kasimbeyli2010}).
\end{remark}

According to our Introduction, the two cones to be separated (namely $-K$ and $A$) do not play symmetric roles, since the separating cone $C_{x^*, \alpha}^{\leq}$ for $\alpha > 0$ is a pointed, closed, convex, dilating cone for $-K$.  In our separation approach, the conditions
\begin{equation}
\label{eq:nec_cond_separation_cones}
    A \cap {\rm cl}({\rm conv}(-K)) = \{0\}
\end{equation}
and
\begin{equation}
\label{eq:nec_cond_separation_cones_2}
    {\rm cl}({\rm conv}\,K) \text{ is pointed}
\end{equation}
are necessary for the strict separation of $-K$ and $A$ by $C_{x^*, \alpha}^{\leq}$. Moreover, the convex sets $S_{-K} = {\rm conv}(-B_{K}) = {\rm conv}(B_{-K})$ and $S_A^0 = {\rm conv}(B_{A} \cup \{0\})$ for the norm-bases  $B_{K}$ and $B_{A}$ of $K$ and $A$, as well as the condition 
\begin{equation}
\label{eq:-clSAcapclSK=empty2}
({\rm cl}\,S_A^0) \cap  ({\rm cl}\, S_{-K}) = \emptyset
\end{equation}
will be of interest in our approach. In particular, \eqref{eq:-clSAcapclSK=empty2} ensures  \eqref{eq:nec_cond_separation_cones}, \eqref{eq:nec_cond_separation_cones_2} and  
\begin{equation}
\label{eq:-clSAcapclSK=emptyConsequence}
0 \notin {\rm cl}\, S_{-K}.
\end{equation}
Indeed, 
$$\eqref{eq:-clSAcapclSK=emptyConsequence} \quad \Longleftrightarrow \quad 0 \notin {\rm cl}\, S_{{\rm cl}({\rm conv}\,K)} \quad\Longrightarrow \quad \eqref{eq:nec_cond_separation_cones_2}$$ by \eqref{eq:0NotInClSconvK} and Theorem \ref{th:0notinclSK} ($1^\circ$, $3^\circ$), while \eqref{eq:-clSAcapclSK=empty2} implies \eqref{eq:-clSAcapclSK=emptyConsequence} and \eqref{eq:nec_cond_separation_cones}, i e.,
\begin{equation}
\label{eq:necessary_sep_cond}
     A \cap {\rm cl}({\rm conv}(-K)) = \{0\}\quad \text{and}\quad 0 \notin {\rm cl}\, S_{-K},
 \end{equation}
since 
\begin{align*}
\eqref{eq:-clSAcapclSK=empty2} \quad 
\Longrightarrow \quad &  {\rm conv}(B_A \cup \{0\}) \cap  ({\rm cl}\, S_{-K}) = \emptyset\\
\Longrightarrow \quad &  (\mathbb{R}_+\cdot B_A) \cap  ({\rm cl}\, S_{-K}) = \emptyset\\
\Longleftrightarrow \quad &  (\mathbb{R}_+\cdot B_A) \cap  (\mathbb{R}_+ \cdot{\rm cl}\, S_{-K}) = \{0\} \quad \text{and} \quad 0 \notin {\rm cl}\, S_{-K}\\
\Longleftrightarrow \quad & \eqref{eq:necessary_sep_cond}.
\end{align*}
From the last two equivalences we can also deduce that \eqref{eq:necessary_sep_cond}
is equivalent to 
\begin{equation}
\label{eq:necessary_sep_cond2}
A \cap  ({\rm cl}\, S_{-K}) = \emptyset.
\end{equation}
Of course, if $A$ is closed and convex, then \eqref{eq:necessary_sep_cond2} $\Longrightarrow$ \eqref{eq:-clSAcapclSK=empty2} since ${\rm cl}\,S_A^0 \subseteq A$, and so \eqref{eq:necessary_sep_cond2} $ \Longleftrightarrow $ \eqref{eq:necessary_sep_cond}  $ \Longleftrightarrow $ \eqref{eq:-clSAcapclSK=empty2}. 
The validity of the implication \eqref{eq:necessary_sep_cond} $\Longrightarrow$ \eqref{eq:-clSAcapclSK=empty2} is studied in more detail in Remark \ref{rem:SepNecCondtions}, Example \ref{ex:counterex_AcapK1} and Theorem \ref{th:separationBanachSpacePureConvex}.
In view of Theorem \ref{th:0notinclSK} ($2^\circ$), under the closedness and convexity of $K$ and $A$, and the validity of the contrary statement of \eqref{eq:-clSAcapclSK=emptyConsequence} (i.e., $0 \in {\rm cl}\, S_{-K}$), we have 
\begin{equation} \label{eq:newEqRelatedToClSetsIntersection}
    ({\rm cl}\,S_A^0) \cap  ({\rm cl}\, S_{-K}) = A \cap  (-K) \cap \mathbb{B} \supseteq \{0\},
\end{equation}
where the last inclusion is an equality if \eqref{eq:nec_cond_separation_cones} is valid.

In the next section, we will prove (under the assumption that ${\rm cl}\,S_A^0$ or ${\rm cl}\, S_{-K}$ is weakly compact) that there is some  
$(x^*, \alpha) \in K^{aw\#}\cap (E^* \times \mathbb{P})$ such that $-K$ and $A$ are strictly separated by $C = C_{x^*, \alpha}^{\leq}$ if (and only if) \eqref{eq:-clSAcapclSK=empty2} is satisfied (see Theorems \ref{th:sep_cones_Banach_1} and \ref{th:sep_cones_Banach_3}).
Figure \ref{fig:cone_separation} visualizes the strict nonlinear separation approach for an example in the real normed space $(\mathbb{R}^2, ||\cdot||_2)$, where $||\cdot||_2$ denotes the Euclidean norm.
\begin{center}
    \begin{figure}[h!]
    \centering
    \resizebox{0.6\hsize}{!}{\input{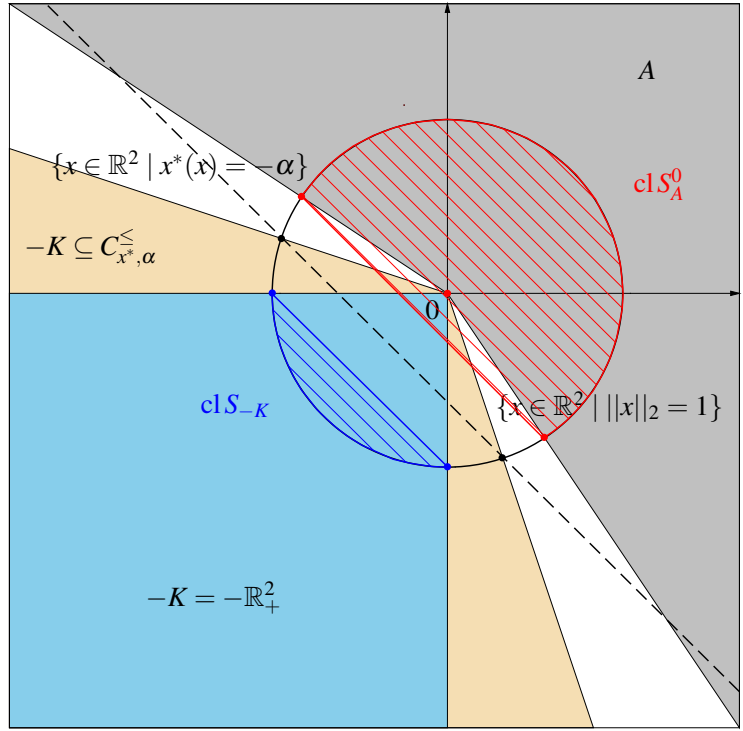}}
    \caption{The nontrivial, closed, pointed, solid, convex cone $-K = \mathbb{R}^2_+$ and the nontrivial, closed, solid, nonconvex cone $A$ in the real normed space $(\mathbb{R}^2, ||\cdot||_2)$ are strictly separated by $C_{x^*, \alpha}^{\leq}$. In particular, the condition \eqref{eq:-clSAcapclSK=empty2} is valid.}
    \label{fig:cone_separation}
\end{figure}
\end{center}

Some more details about the nonlinear cone separation approach can also be found in the paper \cite{GuenKhaTam22}.

\begin{remark}[related to the work \cite{GuenKhaTam22}]
 We would like to point out that a general nonlinear and nonsymmetric separation approach in real topological linear spaces / real locally convex spaces based on a generalization of the function $\varphi_{x^*, \alpha}$ in \eqref{f12} (a combination of a seminorm with a linear function) is developed in G\"unther, Khazayel and Tammer \cite{GuenKhaTam22}. In the present paper, we focus on the particular setting of real (reflexive) normed spaces. However, our main results in Section \ref{sec:cone_separation_normed} can not be derived directly from the results obtained in the more general framework in \cite{GuenKhaTam22} (see Remarks \ref{rem:GKT_general_1},  \ref{rem:GKT_general_2},  \ref{rem:GKT_general_3} and  \ref{rem:GKT_general_4} for details). A major problem is that norm functions defined on an infinite dimensional real normed space are not weakly continuous (as mentioned in Remark \ref{rem:StrongAndWeak3}).
For the  proofs of our results, we use the special structure in real (reflexive) normed spaces (e.g., weak compactness of the unit ball $\mathbb{B}$ in a reflexive Banach space setting, a special linear separation theorem in real normed spaces (Proposition \ref{prop:seperationCompactJahn2}), and the subset $K^{aw\#}$ given by \eqref{eq:defKawSharp} of the augmented dual cone $K^{a+}$). We do not make any compactness assumptions w.r.t. the norm topology.
\end{remark}

\section{Strict Cone Separation Theorems in Real (Reflexive) Normed Spaces} \label{sec:cone_separation_normed}

The aim of this section is to present new strict cone separation theorems in a real (reflexive) normed space $(E, ||\cdot||)$. 
We start with an auxiliary lemma that relates the nonlinear separation of two cones by the functions given in \eqref{f12} with the linear separation of the two norm-bases of the cones. In our new approach to separation, we are using the subset $K^{aw\#}$ defined by \eqref{eq:defKawSharp} of the augmented dual cone $K^{a+}$.

\begin{lemma} \label{lem:nonlinSep_linSep}Let $E$ be a real normed space, $K \subseteq E$ a nontrivial cone with the norm-base $B_{K}$, and $A \subseteq E$ a nontrivial cone with the norm-base $B_{A}$. 
    Then, the following assertions are equivalent:
\begin{itemize}
\item[$1^\circ$] $(x^*, \alpha) \in K^{aw\#}$ satisfies \eqref{sep_non_top_4} (i.e., $-K$ and $A$ are strictly separated by $C = C_{x^*, \alpha}^{\leq}$).
\item[$2^\circ$]  $(x^*, \alpha) \in E^* \times \mathbb{R}_+$ satisfies
    \begin{equation}
    	x^*(a) >  -\alpha > x^*(k)  \quad \mbox{for all } a \in B_A \mbox{ and  }k \in -{\rm cl}_w\, B_K.  \label{sep_non_top_4b}
    \end{equation}	
\end{itemize}
\end{lemma}

\begin{proof}
Consider any $(x^*, \alpha) \in E^* \times \mathbb{R}_+$. It is easy to check that 
\eqref{sep_non_top_4} is equivalent to
\begin{equation}
    \label{sep_non_top_4c}
    x^*(a) + \alpha ||a||  >  0 > x^*(k) + \alpha ||k||   \quad \mbox{for all } a \in B_A \mbox{ and  }k \in - B_K.  
\end{equation}
Furthermore, we have
\begin{align*}
    &  \eqref{sep_non_top_4} \text{ is valid}   \quad \text{and} \quad(x^*, \alpha) \in K^{aw\#} \\
    \iff \quad& \eqref{sep_non_top_4c} \text{ is valid}  \quad \text{and} \quad x^*(-\bar k) > \alpha  \quad   \mbox{ for all } \bar k \in -{\rm cl}_w\, B_K\\
    \iff \quad & x^*(a) + \alpha  >  0 > x^*(k) + \alpha \quad \text{and} \quad  -\alpha > x^*(\bar k) \quad \mbox{ for all } a \in B_A, k \in -B_K, \bar k \in -{\rm cl}_w\, B_K\\
    \iff \quad& x^*(a)  >  - \alpha  > x^*(k) \quad \text{and} \quad  -\alpha > x^*(\bar k) \quad \mbox{ for all } a \in B_A, k \in - B_K, \bar k \in -{\rm cl}_w\, B_K\\
    \iff \quad& x^*(a)  >  - \alpha  > x^*(\bar k) \quad \mbox{ for all } a \in B_A, \bar k \in -{\rm cl}_w\, B_K.
\end{align*}
\end{proof}

\begin{remark}
    Notice that the (nonnegative) sign of the parameter $\alpha$ is very important in our nonsymmetric separation approach. Assertion $2^\circ$ of Lemma \ref{lem:nonlinSep_linSep} without the assumption $\alpha \geq 0$ would provide separating cones that are dilating cones for either $-K$ or $A$, which corresponds to a symmetric separation concept. Within our nonsymmetric separation concept, the separating cone is always a dilating cone for the cone $-K$.

    Notice that condition \eqref{sep_non_top_4b} for some $(x^*, \alpha) \in E^* \times \mathbb{R}_+$ does not imply the (strict linear cone separation) condition  \eqref{eq:sep_lin_cones} 
    in general (a counter example for the case $K$ is convex and $A$ is nonconvex is given in Figure \ref{fig:cone_separation}; see also Theorem \ref{th:separationBanachSpacePureConvex}). 
\end{remark}

In the following theorem, we present sufficient conditions for strict cone separation.

\begin{theorem} \label{th:sep_cones_Banach_1}
	Let $E$ be a real normed space, $K \subseteq E$ a nontrivial cone with the norm-base $B_{K}$, and $A \subseteq E$ a nontrivial cone with the norm-base $B_{A}$.  Assume that one of the sets ${\rm cl}\, S_{-K}$ and ${\rm cl}\, S_A^0$ is weakly compact  (e.g. if $E$ is also reflexive). If \eqref{eq:-clSAcapclSK=empty2} is valid, 
	then there exists $(x^*, \alpha) \in K^{aw\#} \cap (E^* \times \mathbb{P})$ such that \eqref{sep_non_top_4} is valid (i.e., $-K$ and $A$ are strictly separated by $C = C_{x^*, \alpha}^{\leq}$).
\end{theorem}

\begin{proof}
Assume that \eqref{eq:-clSAcapclSK=empty2} holds true.	
Notice that in a reflexive Banach space setting, both sets ${\rm cl}\, S_{-K}$ and ${\rm cl}\, S_A^0$ are weakly compact as closed, convex subsets of the weakly compact unit ball $\mathbb{B}$.
Employing Proposition \ref{prop:seperationCompactJahn2} (the strict linear separation of convex sets in real normed spaces), there exist $x^* \in E^* \setminus \{0\}$ and $\gamma, \beta \in \mathbb{R}$ with
\begin{equation*}
\label{eq:linSepConSetSAandCLSKrefBan}
x^*(a) \geq \beta > \gamma \geq x^*(k)   \quad \mbox{for all } a \in {\rm cl}\,S_A^0 \mbox{ and }k \in {\rm cl}\,S_{-K}.
\end{equation*}
We obtain $\beta \leq 0$ since $0 \in S_A^0$. Fix some $\delta \in (\gamma, \beta) \subseteq (-\infty, 0)$ and put $\alpha := - \delta \; (> 0)$. Since $- {\rm cl}_w\, B_K \subseteq {\rm cl}\,S_{-K}$ we get
\begin{equation*}
\label{eq:linSepConSetSAandCLSKrefBan2}
x^*(a) \geq \beta > -\alpha > \gamma \geq x^*(k)   \quad \mbox{for all } a \in B_A  \mbox{ and }k \in - {\rm cl}_w\, B_K.
\end{equation*}
Lemma \ref{lem:nonlinSep_linSep} (implication $2^\circ \Longrightarrow 1^\circ$) yields the desired strict separation \eqref{sep_non_top_4} and $(x^*, \alpha) \in K^{aw\#}$.
\end{proof}

\begin{remark}[related to the work \cite{GuenKhaTam22}] \label{rem:GKT_general_1}
     Under the assumptions of Theorem \ref{th:sep_cones_Banach_1}, the existence of a pair $(x^*, \alpha) \in K^{a\#} \cap (E^* \times \mathbb{P})$ with 
     \eqref{sep_non_top_4} can also be ensured by applying \cite[Th. 5.9]{GuenKhaTam22} for the real locally convex space $(E, \tau_w)$ taking into account that $(E, ||\cdot||)^* = (E, \tau_w)^*$ , ${\rm cl}_w\,S_A^0 = {\rm cl}\,S_A^0$ and ${\rm cl}_w\,S_{-K} = {\rm cl}\,S_{-K}$. Our above proof of Theorem \ref{th:sep_cones_Banach_1} shows that actually one can ensure that $(x^*, \alpha) \in K^{aw\#} \cap (E^* \times \mathbb{P})$.
\end{remark}

Now, we are able to present our main strict cone separation result. 

\begin{theorem} \label{th:sep_cones_Banach_3}
Let $E$ be a real normed space, $K \subseteq E$ a nontrivial cone with the norm-base $B_{K}$, and $A \subseteq E$ a nontrivial cone with the norm-base $B_{A}$. Assume that ${\rm cl}\,S_{-K}$ is weakly compact (e.g. if $E$ is also reflexive).
Then, the following conditions are equivalent:
	\begin{itemize}
		\item[$1^\circ$] 
		$
		({\rm cl}\,S_A^0) \cap ({\rm cl}\, S_{-K}) = \emptyset.
		$
		\item[$2^\circ$]  There exists $(x^*, \alpha) \in K^{aw\#} \cap (E^* \times \mathbb{P})$  such that \eqref{sep_non_top_4} is valid.
		\item[$3^\circ$]  There exists $(x^*, \alpha) \in K^{aw\#}$ such that  \eqref{sep_non_top_4} is valid.
	\end{itemize}	
\end{theorem}

\begin{proof} 
The implication $1^\circ \Longrightarrow 2^\circ$ is given in Theorem \ref{th:sep_cones_Banach_1}. 
Obviously, $2^\circ \Longrightarrow 3^\circ$ is valid.

It remains to show the implication $3^\circ \Longrightarrow 1^\circ$. Assume that $3^\circ$ is valid. By Lemma \ref{lem:nonlinSep_linSep} (implication $1^\circ \Longrightarrow 2^\circ$), there is $(x^*, \alpha) \in E^* \times \mathbb{R}_+$ such that 
\eqref{sep_non_top_4b} is valid. 
Since $x^*$ is continuous and ${\rm cl}_w\, B_K \, (\subseteq {\rm cl}\,S_{K})$ is weakly compact, there is $\beta < 0$ such that
$$
x^*(a) \geq  -\alpha  > \beta \geq x^*(k) \quad \mbox{for all } a \in B_A \cup \{0\}\mbox{ and  }k \in -B_{K}
$$
by the well-known Weierstraß theorem (see \cite[Th. 3.26]{Jahn2011}). 
By the convexity of closed half spaces and the continuity of $x^*$, we get 
$$
x^*(a)   \geq  -\alpha > \beta \geq x^*(k) \quad \mbox{for all }  a \in {\rm cl}\, S_A^0\mbox{ and  }k \in {\rm cl}\, S_{-K}.$$
As a direct consequence we conclude that
$({\rm cl}\, S_A^0) \cap ({\rm cl}\, S_{-K}) = \emptyset$.
\end{proof}

\begin{remark}
    Observe that because of $1^\circ \iff 2^\circ$ in Theorem \ref{th:sep_cones_Banach_3} the condition $({\rm cl}\,S_A^0) \cap ({\rm cl}\, S_{-K}) = \emptyset$  characterizes the existence of a strict cone separation of the (not necessarily convex) cones $-K$ and $A$ by a separating cone that is given by a Bishop-Phelps cone (which is a nontrivial, closed, pointed, topologically solid, convex cone). 
    In assertion $3^\circ$ of Theorem \ref{th:sep_cones_Banach_3}, the separating cone may also be given by a halfspace cone (if $\alpha = 0$), which is not a Bishop-Phelps type cone. 
    
    Due to the nonsymmetric separation approach,  the condition
    $({\rm cl}\,S_A^0) \cap ({\rm cl}\, S_{-K}) \neq ({\rm cl}\,S_{-K}^0) \cap ({\rm cl}\, S_{A})$
    may happen (see the example given in Figure \ref{fig:cone_separation} with $({\rm cl}\,S_A^0) \cap ({\rm cl}\, S_{-K}) = \emptyset \neq ({\rm cl}\,S_{-K}^0) \cap ({\rm cl}\, S_{A})$). 
    
    Under the assumptions of Theorem \ref{th:sep_cones_Banach_3}, any of the assertions $1^\circ, 2^\circ$ and $3^\circ$ implies the condition \eqref{eq:nec_cond_separation_cones_2}, and so $K$ is necessarily a pointed cone.
\end{remark}

\begin{remark}[related to the work \cite{GuenKhaTam22}] \label{rem:GKT_general_2}
    It is important to mention that, under the assumptions of Theorem \ref{th:sep_cones_Banach_3}, the implication $3^\circ \Longrightarrow 1^\circ$ (respectively, $2^\circ \Longrightarrow 1^\circ$) is only a consequence of \cite[Th. 5.12]{GuenKhaTam22} (applied for the real locally convex space $(E, \tau_w)$) in the case that $B_K$ is weakly compact (e.g. if $||\cdot||$ is weakly continuous (or equivalently, $E$ has finite dimension) and $K$ is closed). Notice, in this case $(E, ||\cdot||)^* = (E, \tau_w)^*$, ${\rm cl}_w\,S_A^0 = {\rm cl}\,S_A^0$, ${\rm cl}_w\,S_{-K} = {\rm cl}\,S_{-K}$, and $0 \notin {\rm cl}\,S_{-K}$ (by Theorem \ref{th:0notinclSK} ($4^\circ$)).     
\end{remark}

\begin{remark}[related to the work \cite{Kasimbeyli2010}] \label{rem:new_fourth_assertion}
In the spirit of Kasimbeyli \cite[Th. 4.3]{Kasimbeyli2010}, one could consider an assertion  
\begin{equation}
\label{ass_four}
    \text{``There exists $(x^*, \alpha) \in K^{a\#}$ such that  \eqref{sep_non_top_4} is valid.''}
\end{equation}
(notice that here $\alpha = 0$ may hold).
However, under the assumptions of Theorem \ref{th:sep_cones_Banach_3} and a not weakly compact norm-base $B_K$, this assertion \eqref{ass_four} does not imply the assertion $1^\circ$ of Theorem \ref{th:sep_cones_Banach_3} in general (even under the convexity and closedness of both cones $K$ and $A$, and the reflexivity of $E$), as seen in the next Example \ref{ex:counter_example_for_Kasimbeyli_theorem}. Clearly, if $B_{K}$ is weakly compact, then all the assertions in Theorem \ref{th:sep_cones_Banach_3} are equivalent to \eqref{ass_four} (since $K^{aw\#} = K^{a\#}$).
\end{remark}

\begin{example} \label{ex:counter_example_for_Kasimbeyli_theorem}
Consider again Example \ref{ex:counter_example_0inclSK} with the natural ordering cone $K$ given by \eqref{eq:standard_cone_l2} in $\ell_2$ (endowed with the inner product $\langle \cdot, \cdot \rangle$) and the sequence 
$x^* = (\frac{1}{n})_{n \in \mathbb{N}}$. Moreover, we consider the (ray) cone given by
$$
A := \mathbb{R}_+ \cdot \{x^*\} \quad \mbox{ with the norm-base } \quad B_A = \left\{\frac{x^*}{||x^*||_{\ell_2}}\right\}.
$$
Notice that $A$ is nontrivial, pointed, closed and convex, and $A \setminus \{0\} = \mathbb{P} \cdot \{x^*\}$.
As already observed in Example \ref{ex:counter_example_0inclSK}, $0 \in {\rm cl}_w\, B_K \subseteq {\rm cl}\, S_K$ and so, $0 \in -{\rm cl}\, S_K = {\rm cl}\, S_{-K}$. Since $K$ and $A$ are closed and convex with $A \cap (-K) = \{0\}$, we conclude (by the discussion related to \eqref{eq:newEqRelatedToClSetsIntersection})
$$
({\rm cl}\,S_A^0) \cap ({\rm cl}\, S_{-K}) = A \cap (-K) \cap \mathbb{B} = \{0\} \neq \emptyset.
$$
Thus, $1^\circ$ of Theorem \ref{th:sep_cones_Banach_3} is false in this example. Of course, by Theorem \ref{th:sep_cones_Banach_3} both assertions $2^\circ$ and $3^\circ$ are false too. Actually by Theorem \ref{th:0notinclSK} ($3^\circ$), we have
$$
K^{a+} \cap (E^* \times \mathbb{P}) = K^{a\#} \cap (E^* \times \mathbb{P}) = K^{aw\#} \cap (E^* \times \mathbb{P}) =  K^{aw\#} = \emptyset.
$$
However, the assertion \eqref{ass_four} is true. Indeed, for $(x^*, 0) \in K^{a\#}$, we get 
\begin{align*}
     \langle x^*, s x^* \rangle =  s ||x^*||_{\ell_2}^2 > 0 > \langle x^*, k \rangle  \quad \mbox{ for all }  s \in \mathbb{P} \text{ and } k \in -K \setminus \{0\},
\end{align*}
which shows that \eqref{sep_non_top_4} is valid.
\end{example}

\begin{lemma} \label{lem:new_lemma_strict_sep_conditions}
Let $E$ be a real normed space, $K \subseteq E$ a nontrivial cone with the norm-base $B_{K}$, and $A \subseteq E$ a nontrivial cone with the norm-base $B_{A}$. 
Then, the following assertions are equivalent:
   \begin{itemize}
		\item[$1^\circ$]  There exists $(x^*, \alpha) \in {\rm cor}\,K^{a+}$ such that \eqref{sep_non_top_4} is valid.
        \item[$2^\circ$]  There exists $(x^*, \alpha) \in {\rm cor}\,K^{a+}$ such that \eqref{sep_non_top_4} with ${\rm cl}({\rm conv}\,K)$ in the role of $K$ and ${\rm cl}\, A$ in the role of $A$ is valid.  
        \item[$3^\circ$]  There exist $\delta_2 > \delta_1 > 0$ and $x^* \in E^*$ such that, for any $\alpha \in (\delta_1, \delta_2)$, we have $(x^*, \alpha) \in {\rm cor}\,K^{a+}$, and \eqref{sep_non_top_4} with ${\rm cl}({\rm conv}\,K)$ in the role of $K$ and ${\rm cl}\, A$ in the role of $A$ is valid.  
        \item[$4^\circ$] There exist $\delta_2 > \delta_1 > 0$ and $x^* \in E^*$ such that, for any $\alpha \in (\delta_1, \delta_2)$, we have $(x^*, \alpha) \in K^{a\#}$, and \eqref{sep_non_top_4} with ${\rm cl}({\rm conv}\,K)$ in the role of $K$ and ${\rm cl}\, A$ in the role of $A$ is valid.  
   \end{itemize}  
\end{lemma}

\begin{proof} 
Obviously, $3^\circ \Longrightarrow 2^\circ \Longrightarrow 1^\circ$, and $3^\circ \Longrightarrow 4^\circ$ (by Theorem  \ref{th:solidness_augmented_dual_cones} ($1^\circ$)).

Let us show the implication $4^\circ \Longrightarrow 3^\circ$. Assume that $4^\circ$ is valid. For any $\delta \in (\alpha, \delta_2)$, we have
$x^*(k) > \delta > \alpha$ for $k \in B_K$, and so $\inf_{x  \in B_K}\,  x^*(x) > \alpha$. By Theorem  \ref{th:solidness_augmented_dual_cones} ($1^\circ$) we conclude $(x^*, \alpha) \in {\rm cor}\,K^{a+}$. Of course, this shows $3^\circ$.

It remains to show that $1^\circ \Longrightarrow 3^\circ$. Assume that $1^\circ$ is valid, i.e., there is $(x^*, \alpha) \in {\rm cor}\,K^{a+}$ such that \eqref{sep_non_top_4} is valid. Then, Theorem \ref{th:solidness_augmented_dual_cones} ($1^\circ$) yields $\beta := \inf_{x  \in B_K}\,  x^*(x) > \alpha$, and $(x^*, \gamma) \in {\rm cor}\,K^{a+} \subseteq K^{a\#}$ for $\gamma \in (0,\beta)$. Consider $\delta \in (\alpha, \beta)$; hence $(x^*, \delta) \in {\rm cor}\,K^{a+}$. Since $x^*(k) \geq \beta$ for all $k \in B_K$, one has $\varphi_{-x^*, \beta}(k) \leq 0$ for all $k \in K$, and so $K \subseteq C_{-x^*, \beta}^{\leq}$. Because $C_{-x^*, \beta}^{\leq}$ is a closed, convex set, one obtains that $K \subseteq {\rm cl}({\rm conv}\,K) \subseteq  C_{-x^*, \beta}^{\leq}$, and so 
$$
x^*(-\hat k) + \delta ||\hat k|| < x^*(-\hat k) + \beta ||\hat k|| \leq 0 \quad \text{for all } \hat k \in ({\rm cl}({\rm conv}\,K)) \setminus \{0\}.
$$
Because $x^*(a) + \alpha ||a|| > 0$ for $a \in A \setminus \{0\}$ by \eqref{sep_non_top_4}, one has $x^*(a) + \alpha ||a|| \geq 0$
for $a \in {\rm cl}\,A$, and so 
$$x^*(a) + \delta ||a|| > 0 \quad \text{for all } a \in ({\rm cl}\,A) \setminus \{0\}.$$
We conclude the validity of assertion $3^\circ$.
\end{proof}

As an immediate consequence of Theorem \ref{th:solidness_augmented_dual_cones}, Theorem \ref{th:sep_cones_Banach_3} and Lemma \ref{lem:new_lemma_strict_sep_conditions} we have the following result.

\begin{corollary}
    Let $E$ be a real normed space, $K \subseteq E$ a nontrivial cone with the norm-base $B_{K}$, and $A \subseteq E$ a nontrivial cone with the norm-base $B_{A}$. Assume that ${\rm cl}\, S_{-{\rm cl}({\rm conv}\,K)}$ is weakly compact  (e.g. if $E$ is also reflexive). Then, the following assertions are equivalent:  
    \begin{itemize}
        \item[$1^\circ$] $({\rm cl}\,S_A^0) \cap ({\rm cl}\, S_{-K}) = \emptyset$.
        \item[$2^\circ$] $({\rm cl}\,S_{{\rm cl}\, A}^0) \cap ({\rm cl}\, S_{-{\rm cl}({\rm conv}\,K)}) = \emptyset$.
        \item[$3^\circ$] $-K$ and $A$ are strictly separated by $C = C_{x^*, \alpha}^{\leq}$ for some $(x^*, \alpha) \in K^{aw\#} \cap (E^* \times \mathbb{P})$.
        \item[$4^\circ$] $-{\rm cl}({\rm conv}\,K)$ and ${\rm cl}\, A$ are strictly separated by $C = C_{x^*, \alpha}^{\leq}$ for some $(x^*, \alpha) \in K^{aw\#} \cap (E^* \times \mathbb{P})$.
    \end{itemize}
\end{corollary}

\begin{proof} The weak compactness of ${\rm cl}\, S_{-{\rm cl}({\rm conv}\,K)}$ ensures the weak compactness of ${\rm cl}_w\, B_{{\rm cl}({\rm conv}\,K)}$, ${\rm cl}\, S_{-K}$ and ${\rm cl}_w\, B_{K}$. Taking into account that $K^{a+} = ({\rm cl}({\rm conv}\,K))^{a+}$ (as mentioned in \eqref{eq:K+=convKa+=clvonvKa+}), by Theorem \ref{th:solidness_augmented_dual_cones} ($2^\circ$) we get 
$({\rm cl}({\rm conv}\,K))^{aw\#} \cap (E^* \times \mathbb{P}) = {\rm cor}\,K^{a+} =  K^{aw\#} \cap (E^* \times \mathbb{P})$.
Then, the result follows by Theorem \ref{th:sep_cones_Banach_3} (applied for $K$ and $A$ as well as for ${\rm cl}({\rm conv}\,K)$ and ${\rm cl}\, A$) and by Lemma \ref{lem:new_lemma_strict_sep_conditions} (equivalence $1^\circ \iff 2^\circ$).
\end{proof}

\begin{remark}
\label{rem:SepNecCondtions}
Suppose that the assumptions of Theorem \ref{th:sep_cones_Banach_3} are valid. Consider some $(x^*, \alpha) \in K^{aw\#}$. Then, \eqref{eq:necessary_sep_cond} (i.e., the conditions \eqref{eq:nec_cond_separation_cones} and \eqref{eq:-clSAcapclSK=emptyConsequence}; or equivalently,  \eqref{eq:necessary_sep_cond2}) is necessary for the fact that $-K$ and $A$ are strictly separated by $C = C_{x^*, \alpha}^{\leq}$ (or equivalently, \eqref{eq:-clSAcapclSK=empty2} is valid); see Section \ref{sec:nonSepAppAndBPCones}. However, \eqref{eq:necessary_sep_cond} does not imply the condition \eqref{eq:-clSAcapclSK=empty2} (so \eqref{eq:necessary_sep_cond} is necessary but not sufficient for strict separation in the above sense). This can be seen in the next example.
\end{remark}

\begin{example} \label{ex:counterex_AcapK1}
	Consider the reflexive Banach space $(\mathbb{R}^2, ||\cdot||_1)$, where $||\cdot||_1$  denotes the Manhattan norm  (also known as $\ell_1$-norm). Define $x^{(\lambda)} := (1- \lambda, \lambda) \in \mathbb{R}^2$ for all $\lambda \in [0,1]$ and put $\Omega^1 := \mathbb{R}_+ \cdot {\rm conv}(\{x^{(0)}, x^{(\frac{1}{5})}\})$, $\Omega^2 := \mathbb{R}_+ \cdot {\rm conv}(\{x^{(\frac{2}{5})}, x^{(\frac{3}{5})}\})$, $\Omega^3 := \mathbb{R}_+ \cdot {\rm conv}(\{x^{(\frac{4}{5})}, x^{(1)}\})$. Clearly, $K := -\Omega^2$ and $A := \Omega^1 \cup \Omega^3$ are nontrivial, pointed, solid, closed cones, $K$ is convex but $A$ is nonconvex. Now, it is easy to check that $A \cap {\rm cl}({\rm conv}(-K)) = \{0\}$ and $0 \notin {\rm cl}\, S_{-K}$ are valid but 
	$({\rm cl}\,S_A^0) \cap ({\rm cl}\, S_{-K}) \neq \emptyset$.
\end{example}

However, under the closedness and convexity of the nontrivial cone $A$, one can prove that \eqref{eq:-clSAcapclSK=empty2} and \eqref{eq:necessary_sep_cond} are equivalent. 

\begin{theorem} \label{th:separationBanachSpacePureConvex} Let $E$ be a real normed space, $K \subseteq E$ a nontrivial cone with the norm-base $B_{K}$, and $A \subseteq E$ a nontrivial, closed, convex cone with the norm-base $B_{A}$. Assume that ${\rm cl}\, S_{-K}$ is weakly compact (e.g. if $E$ is also reflexive).
Then, the following conditions are equivalent:
\begin{itemize}
    \item[$1^\circ$] $({\rm cl}\,S_A^0) \cap ({\rm cl}\, S_{-K}) = \emptyset$.		
    \item[$2^\circ$]  There exists $(x^*, \alpha) \in K^{aw\#} \cap (E^* \times \mathbb{P})$  such that \eqref{sep_non_top_4} is valid.
    \item[$3^\circ$]  There exists $(x^*, \alpha) \in K^{aw\#}$ such that  \eqref{sep_non_top_4} is valid.
    \item[$4^\circ$] $A \cap ({\rm cl}({\rm conv}(-K))) = \{0\}$ and $0 \notin {\rm cl}\, S_{-K}$.
\end{itemize}
If the set ${\rm cl}\,S_{{\rm cl}({\rm conv}\, K)}$ is weakly compact (e.g., if $E$ is reflexive), then any of the assertions $1^\circ$ -- $4^\circ$ is equivalent to
\begin{itemize}
    \item[$5^\circ$] $0 \notin {\rm cl}\, S_{-K}$ and there exists $x^* \in E^* \setminus \{0\}$ such that \eqref{eq:sep_lin_cones} with ${\rm cl}({\rm conv}\,K)$ in the role of $K$ is valid.
\end{itemize}
\end{theorem}

\begin{proof} Theorem \ref{th:sep_cones_Banach_3} shows that $1^\circ \iff 2^\circ \iff 3^\circ$. As mentioned in Remark \ref{rem:SepNecCondtions}, we have  $1^\circ \Longrightarrow 4^\circ$. The implication $4^\circ \Longrightarrow 1^\circ$ follows by the fact that \eqref{eq:necessary_sep_cond} $\Longleftrightarrow$ \eqref{eq:necessary_sep_cond2} and ${\rm cl}\,S_A^0 \subseteq A$ by the closedness and convexity of $A$ (see Section \ref{sec:nonSepAppAndBPCones}).

Now, assume that ${\rm cl}\,S_{{\rm cl}({\rm conv}\, K)}$ is weakly compact.
The implication $4^\circ \Longrightarrow 5^\circ$ is a direct consequence of \eqref{eq:0NotInClSconvK} and 
Theorem \ref{th:0notinclSK} ($3^\circ$), which ensure
$0  \notin {\rm cl}\,S_{{\rm cl}({\rm conv}\,K)} \iff 0 \notin {\rm cl}\,S_{{\rm cl}({\rm conv}(-K))} \iff 0 \notin {\rm cl}\,S_{-K}$, 
and the linear separation result in Proposition \ref{prop:Jahn_sep_two_closed_cones}. 
The remaining implication $5^\circ \Longrightarrow 4^\circ$ is obvious, since \eqref{eq:sep_lin_cones} (with ${\rm cl}({\rm conv}\,K)$ in the role of $K$) $\Longrightarrow$ $A \cap ({\rm cl}({\rm conv}(-K)) = \{0\}$.
\end{proof}

\begin{remark}[related to the work \cite{GuenKhaTam22}]               \label{rem:GKT_general_3}
    Under the assumptions of Theorem \ref{th:separationBanachSpacePureConvex}, if $||\cdot||$ is weakly continuous (or equivalently, $E$ has finite dimension), and $K$ is closed and convex, then the equivalences of assertions $1^\circ$, $2^\circ$, $4^\circ$ and $5^\circ$ in Theorem \ref{th:separationBanachSpacePureConvex} can also be concluded by \cite[Th. 5.16 and Rem. 5.17]{GuenKhaTam22} (applied for the real locally convex space $(E, \tau_w)$).
\end{remark}

\begin{remark}
The condition $0 \notin {\rm cl}\, S_{-K}$ in Theorem \ref{th:separationBanachSpacePureConvex} ($4^\circ$) is essential, as Example \ref{ex:counter_example_for_Kasimbeyli_theorem} shows (there $K$ and $A$ are nontrivial, pointed, closed, convex cones in $E$ with $A \cap (-K) = \{0\}$, $0 \in {\rm cl}\, S_{-K}$ and $({\rm cl}\,S_A^0) \cap ({\rm cl}\, S_{-K}) \neq \emptyset$).
\end{remark}

We get the following result (cf. Kasimbeyli \cite[Th. 4.3]{Kasimbeyli2010}) as a direct consequence of our main Theorem \ref{th:sep_cones_Banach_3}. 

\begin{proposition}\label{cor:sep_main_top_kasimbeyli}
	Assume that $E$ is a real normed space, $K \subseteq E$ is a nontrivial cone with the norm-base $B_{K}$, and  $A \subseteq E$ is a nontrivial cone with the norm-base $B_{A}$. Suppose that ${\rm cl}\, S_{-K}$ is weakly compact (e.g. if $E$ is also reflexive). Then, the following assertions are equivalent:
	\begin{itemize}
		\item[$1^\circ$] 
		$({\rm cl}\,S_{{\rm bd}\,A}^0) \cap ({\rm cl}\, S_{-K}) = \emptyset$.
		\item[$2^\circ$]  
        There exists $(x^*, \alpha) \in K^{aw\#} \cap (E^* \times \mathbb{P})$  such that
		\begin{align}
		x^*(a) + \alpha ||a|| &  >  0 > x^*(k) + \alpha ||k||  \quad \mbox{ for all } a \in ({\rm bd}\,A)\setminus \{0\},\,k \in -K \setminus \{0\}.  \label{sep_non_top_6} 
		\end{align}
		\item[$3^\circ$]   
        There exists $(x^*, \alpha) \in K^{aw\#}$  such that \eqref{sep_non_top_6} is valid.
	\end{itemize}	
\end{proposition}

\begin{remark}[related to the work \cite{GuenKhaTam22}] \label{rem:GKT_general_4}
 Under the assumptions of Proposition \ref{cor:sep_main_top_kasimbeyli}, if $B_K$ is weakly compact (e.g. if $||\cdot||$ is weakly continuous (or equivalently, $E$ has finite dimension), and $K$ is closed), then the equivalence of the assertions $1^\circ$ and $2^\circ$ in Proposition \ref{cor:sep_main_top_kasimbeyli} follows also from \cite[Cor. 5.15]{GuenKhaTam22} (applied for the real locally convex space $(E, \tau_w)$).
\end{remark}

\begin{remark}[related to the work \cite{Kasimbeyli2010}] \label{rem:Kasimbeyli_paper}
    If, in addition, $A$ is closed, then
    $S_{{\rm bd}\,A}^0 = {\rm conv}(B_{{\rm bd}\,A} \cup \{0\}) = {\rm conv}((B_{A} \cap {\rm bd}\,A) \cup \{0\})$.
    The latter set is given in Kasimbeyli \cite[Th. 4.3]{Kasimbeyli2010}.  
    In comparison to the separation result derived in \cite[Th. 4.3]{Kasimbeyli2010}, our Proposition \ref{cor:sep_main_top_kasimbeyli} assumes some weaker conditions (concerning convexity and closedness requirements) w.r.t. the involved cones. It is important to mention,  Example \ref{ex:counter_example_for_Kasimbeyli_theorem} also shows that, under the assumptions of Proposition \ref{cor:sep_main_top_kasimbeyli}, one can not replace the set $K^{aw\#}$ by the upper set $K^{a\#}$ in assertion $3^\circ$ of Proposition \ref{cor:sep_main_top_kasimbeyli} in general (even under the convexity and closedness of both cones $K$ and $A$, and the reflexivity of $E$). Thus, one implication in the cone separation result by Kasimbeyli \cite[Th. 4.3]{Kasimbeyli2010} is invalid in the case that $B_K$ is not weakly compact.
    However, under the weak compactness of $B_{K}$ (e.g. if $E$ has finite dimension and $K$ is closed) one can replace the set $K^{aw\#}$ by $K^{a\#}$ in assertion $3^\circ$ of Proposition \ref{cor:sep_main_top_kasimbeyli} (since $K^{aw\#} = K^{a\#}$).
\end{remark}

\begin{corollary} \label{cor:GMPP}
    Assume that $E$ is a real normed space, $K \subseteq E$ is a nontrivial cone with the norm-base $B_{K}$, and  $A \subseteq E$ is a nontrivial cone with the norm-base $B_{A}$. Suppose that ${\rm cl}\, S_{-K}$ is weakly compact  (e.g. if $E$ is also reflexive). Then, the following assertions are equivalent:
    \begin{itemize}
        \item[$1^\circ$] 
		$({\rm cl}\,S_{{\rm bd}\,A}^0) \cap ({\rm cl}\, S_{-K}) = \emptyset$.
        \item[$2^\circ$]  
        There exist $\delta_2 > \delta_1 > 0$ and $x^* \in E^*$ such that, for any $\alpha \in (\delta_1, \delta_2)$, we have $(x^*, \alpha) \in K^{a\#}$, and \eqref{sep_non_top_6} with ${\rm cl}({\rm conv}\,K)$ in the role of $K$ is valid.  
   \end{itemize}
\end{corollary}
    
\begin{proof}
    This result is a direct consequence of Theorem \ref{th:solidness_augmented_dual_cones} ($2^\circ$), 
    Lemma \ref{lem:new_lemma_strict_sep_conditions} (applied for ${\rm bd}\,A$ in the role of $A$) and Proposition \ref{cor:sep_main_top_kasimbeyli}.
\end{proof}

\begin{remark}[related to the work \cite{CastanoEtAl2023}]            \label{rem:new_paper}
    The result in Corollary \ref{cor:GMPP} can also be derived (under the additional assumption of the reflexivity of $E$) by Proposition \ref{cor:sep_main_top_kasimbeyli} and the separation result by Garc\'{i}a-Casta\~{n}o, Melguizo-Padial and Parzanese \cite[Th. 3.1,  (i) $\Leftrightarrow$ (iii)]{CastanoEtAl2023} ($E$ being reflexive, $1^\circ \Leftrightarrow (K,-A)$ has the strict separation property SSP in the sense of \cite[Def. 2]{CastanoEtAl2023}). Our Lemma \ref{lem:new_lemma_strict_sep_conditions}, which is proven with the aid of
   Theorem \ref{th:solidness_augmented_dual_cones} ($1^\circ$), even more shows that \cite[Th. 3.1]{CastanoEtAl2023} (which is formulated in a normed setting) is based on elements in the algebraic interior of the augmented dual cone. Thus, in view of our derived results, one can understand the relationships between \cite[Th. 4.3]{Kasimbeyli2010} (a formulation with a specific $\alpha$) and \cite[Th. 3.1]{CastanoEtAl2023} (a formulation with an $\alpha$-interval). 
\end{remark}

Notice that our main Theorems \ref{th:sep_cones_Banach_1} and \ref{th:sep_cones_Banach_3} are formulated more generally for two cones $A$ and $K$ and the condition \eqref{sep_non_top_4} is involved. In contrast, Proposition \ref{cor:sep_main_top_kasimbeyli} and the results by Kasimbeyli \cite[Th. 4.3]{Kasimbeyli2010} and by Garc\'{i}a-Casta\~{n}o, Melguizo-Padial and Parzanese \cite[Th. 3.1]{CastanoEtAl2023} consider two cones $A$ and $K$ and the condition \eqref{sep_non_top_6}. Clearly, if $A$ is closed, then \eqref{sep_non_top_4} implies  \eqref{sep_non_top_6}.

\section{Conclusions} \label{sec:conclusion}

It is well known that cone separation theorems are important in different fields of mathematics (such as variational analysis, convex analysis, convex geometry, optimization). 
In this paper, we derived new strict cone separation results in the real (reflexive) normed space setting (see Theorem \ref{th:sep_cones_Banach_1}, Theorem \ref{th:sep_cones_Banach_3} and Proposition \ref{cor:sep_main_top_kasimbeyli})  following the nonlinear and nonsymmetric separation approach by Kasimbeyli \cite{Kasimbeyli2010}, which is based on augmented dual cones and Bishop-Phelps type (normlinear) separating functions.  
In our considerations, we established the relationships of our results with the result by
Kasimbeyli \cite{Kasimbeyli2010} and the recent results by
Garc\'{i}a-Casta\~{n}o, Melguizo-Padial and Parzanese \cite{CastanoEtAl2023} and G\"unther, Khazayel and Tammer \cite{GuenKhaTam22}. In particular, as key tools for deriving our new strict cone separation theorems we used
\begin{itemize}
    \item[$\bullet$] the set $K^{aw\#}$ given by \eqref{eq:defKawSharp} (for a nontrivial, pointed cone $K$ in a real normed space with the norm-base $B_K$), which is a subset of the well known set $K^{a\#}$ (as considered in \cite{Kasimbeyli2010}) 
and is nonempty (under the assumption that ${\rm cl}_w\,B_K$ is weakly compact) if and only if $0 \notin {\rm cl}({\rm conv}(B_K))$ (see Theorem \ref{th:0notinclSK}),
\item[$\bullet$] a new characterization of the algebraic interior of augmented dual cones in real (reflexive) normed spaces (see Theorem \ref{th:solidness_augmented_dual_cones}),
\item[$\bullet$] a classical linear separation theorem for convex sets (see Proposition \ref{prop:seperationCompactJahn2}).
\end{itemize}

We aim to use our strict cone separation theorems for deriving some new scalarization results (based on the conic scalarization method by Kasimbeyli \cite{Kasimbeyli2010}, \cite{Kasimbeyli2013}) for vector optimization problems in real (reflexive) normed spaces. 
Moreover, we are trying to employ our derived cone separation theorems in order to develop separation theorems for (not necessarily convex) sets without cone properties in real (reflexive) normed spaces. 

\section*{Acknowledgment} 

The authors are very grateful to Constantin Z\v{a}linescu for his valuable advice and stimulating discussions, which significantly improved the quality of the paper.



\end{document}